\documentclass[preprint,review,12pt]{elsarticle}
\usepackage[T1]{fontenc}
\usepackage[latin9]{inputenc}
\usepackage{verbatim}
\usepackage{url}
\usepackage{amsmath}
\usepackage{amsthm}
\usepackage{amssymb}
\usepackage{stmaryrd}
\usepackage{graphicx}
\usepackage{xargs}[2008/03/08]

\makeatletter

\providecommand{\tabularnewline}{\\}

\theoremstyle{plain}
\newtheorem{thm}{\protect\theoremname}
\ifx\proof\undefined
\newenvironment{proof}[1][\protect\proofname]{\par
\normalfont\topsep6\p@\@plus6\p@\relax
\trivlist
\itemindent\parindent
\item[\hskip\labelsep\scshape #1]\ignorespaces
}{%
\endtrivlist\@endpefalse
}
\providecommand{\proofname}{Proof}
\fi
\theoremstyle{plain}
\newtheorem{lem}[thm]{\protect\lemmaname}
\theoremstyle{remark}
\newtheorem{rem}[thm]{\protect\remarkname}

\usepackage{ifthen}
\usepackage{url}	     

\usepackage{graphicx}
\usepackage{xcolor}
\usepackage{import}

\usepackage{caption}
\DeclareUrlCommand\email{\urlstyle{tt}\scriptsize}
\urldef{\figueredo}{\email}{figueredo@ieee.org}
\urldef{\kussaba}{\email}{kussaba@lara.unb.br}
\urldef{\ishihara}{\email}{ishihara@ene.unb.br}
\urldef{\badorno}{\email}{adorno@ufmg.br}

\makeatletter
\newcommand*\textfigsize{%
  \@setfontsize\textfigsize{7.0}{7.0}%
}
\makeatother

\makeatletter
\journal{Journal of the Franklin Institute (DOI: doi:10.1016/j.jfranklin.2017.01.028)}

\def\ps@pprintTitle{%
     \let\@oddhead\@empty
     \let\@evenhead\@empty
     \def\@oddfoot{\footnotesize\itshape
       Submitted to \ifx\@journal\@empty Elsevier  
       \else\@journal\fi\hfill}%
     \let\@evenfoot\@oddfoot}

\makeatother

\makeatother

\providecommand{\lemmaname}{Lemma}
\providecommand{\remarkname}{Remark}
\providecommand{\theoremname}{Theorem}

\begin{document}
\global\long\def\dualquaternion#1{\underline{\boldsymbol{#1}}}

\global\long\def\quaternion#1{\boldsymbol{#1}}

\global\long\def\dq#1{\underline{\boldsymbol{#1}}}

\global\long\def\quat#1{\boldsymbol{#1}}

\global\long\def\mymatrix#1{\boldsymbol{#1}}

\global\long\def\myvec#1{\boldsymbol{#1}}

\global\long\def\mapvec#1{\boldsymbol{#1}}

\global\long\def\dualvector#1{\underline{\boldsymbol{#1}}}

\global\long\def\dual{\varepsilon}

\global\long\def\dotproduct#1{\langle#1\rangle}

\global\long\def\norm#1{\left\Vert #1\right\Vert }

\global\long\def\mydual#1{\underline{#1}}

\global\long\def\hamilton#1#2{\overset{#1}{\operatorname{\mymatrix H}}\left(#2\right)}

\global\long\def\hami#1{\overset{#1}{\operatorname{\mymatrix H}}}

\global\long\def\tplus{\dq{{\cal T}}}

\global\long\def\getp#1{\operatorname{\mathcal{P}}\left(#1\right)}

\global\long\def\getd#1{\operatorname{\mathcal{D}}\left(#1\right)}

\global\long\def\swap#1{\text{swap}\{#1\}}

\global\long\def\imi{\hat{\imath}}

\global\long\def\imj{\hat{\jmath}}

\global\long\def\imk{\hat{k}}

\global\long\def\real#1{\operatorname{\mathrm{Re}}\left(#1\right)}

\global\long\def\imag#1{\operatorname{\mathrm{Im}}\left(#1\right)}

\global\long\def\imvec{\boldsymbol{\imath_{m}}}

\global\long\def\vector{\operatorname{vec}}

\global\long\def\mathpzc#1{\fontmathpzc{#1}}

\global\long\def\cost#1#2{\underset{\text{#2}}{\operatorname{\text{cost}}}\left(\ensuremath{#1}\right)}

\global\long\def\diag#1{\operatorname{diag}\left(#1\right)}

\global\long\def\trace#1{\operatorname{tr}\left(#1\right)}

\global\long\def\skewsymproduct#1{\ensuremath{\left\lfloor #1\right\rfloor _{\times}}}

\global\long\def\axisangle#1#2{\ensuremath{R\left(\myvec{#1},#2\right)}}

\global\long\def\closedballset{\ensuremath{\mathbb{B}}}

\global\long\def\complexset{\ensuremath{\mathbb{C}}}

\global\long\def\realset{\ensuremath{\mathbb{R}}}

\global\long\def\rationalset{\ensuremath{\mathbb{Q}}}

\global\long\def\integerset{\ensuremath{\mathbb{Z}}}

\global\long\def\naturalset{\ensuremath{\mathbb{N}}}

\newcommandx\projspace[2][usedefault, addprefix=\global, 1=n]{\mathbb{P}^{#1}\left(#2\right)}

\newcommandx\GL[2][usedefault, addprefix=\global, 1=\realset]{\ensuremath{\textrm{GL}\left(#2,#1\right)}}

\newcommandx\SL[2][usedefault, addprefix=\global, 1=\realset]{\ensuremath{\textrm{SL}\left(#2,#1\right)}}

\global\long\def\SE#1{\ensuremath{\textrm{SE}(#1)}}

\global\long\def\SO#1{\ensuremath{\textrm{SO}(#1)}}

\global\long\def\SU#1{\ensuremath{\textrm{SU}(#1)}}

\newcommandx\liealgebraGL[2][usedefault, addprefix=\global, 1=\realset]{\ensuremath{\mathfrak{gl}\left(#2,#1\right)}}

\newcommandx\liealgebraSL[2][usedefault, addprefix=\global, 1=\realset]{\ensuremath{\mathfrak{sl}\left(#2,#1\right)}}

\global\long\def\liealgebraSE#1{\ensuremath{\mathfrak{se}\left(#1\right)}}

\global\long\def\liealgebraSO#1{\ensuremath{\mathfrak{so}\left(#1\right)}}

\global\long\def\liealgebraSU#1{\ensuremath{\mathfrak{su}\left(#1\right)}}

\global\long\def\sgn{\operatorname{sgn}}

\global\long\def\dualmatrix#1{\left\llbracket \left\llbracket #1\right\rrbracket \right\rrbracket }

\global\long\def\quatset{\ensuremath{\mathbb{H}}}
 %

\global\long\def\unitquatgroup{\ensuremath{\text{Spin}(3)}}
 %

\global\long\def\unitquatset{\mathcal{S}^{3}}
 %

\global\long\def\dualquatset{\mathbb{H}\otimes\mathbb{D}}
 %
{} 

\global\long\def\unitdualquatgroup{\text{Spin}(3)\ltimes\mathbb{R}^{3}}

\global\long\def\unitdualquatset{\dq S}
{}

\begin{frontmatter}{}

\title{Hybrid Kinematic Control for Rigid Body Pose Stabilization using
Dual Quaternions}

\tnotetext[t1]{\copyright 2017. This manuscript version is made available under the CC BY-NC-ND 4.0 license \url{http://creativecommons.org/licenses/by-nc-nd/4.0/}. Supplementary material associated with this article can be found, in the online version, at 10.1016/j.jfranklin.2017.01.028.}

\author[unb]{Hugo T. M. Kussaba\corref{cor1}}

\ead{kussaba@lara.unb.br}

\author[unb]{Luis F. C. Figueredo }

\author[unb]{João Y. Ishihara\fnref{ucla} }

\author[ufmg]{Bruno V. Adorno }

\cortext[cor1]{Corresponding author}

\fntext[ucla]{A portion of this work was completed while the author was with Department
of Electrical Engineering, University of California, Los Angeles (UCLA)
\textendash{} 90095, California }

\address[unb]{Department of Electrical Engineering, University of Brasília (UnB)
\textendash{} 70910-900, Brasília, DF, Brazil}

\address[ufmg]{Department of Electrical Engineering, Federal University of Minas
Gerais (UFMG) \textendash{} 31270-010, Belo Horizonte, MG, Brazil
}
\begin{abstract}
In this paper, we address the rigid body pose stabilization problem
using dual quaternion formalism. We propose a hybrid control strategy
to design a switching control law with hysteresis in such a way that
the global asymptotic stability of the closed-loop system is guaranteed
and such that the global attractivity of the stabilization pose does
not exhibit chattering, a problem that is present in all discontinuous-based
feedback controllers. Using numerical simulations, we illustrate the
problems that arise from existing results in the literature\textemdash as
unwinding and chattering\textemdash and verify the effectiveness of
the proposed controller to solve the robust global pose stability
problem. 
\end{abstract}

\end{frontmatter}{}

\section{Introduction}

Rigid body motion and its control have been extensively investigated
in the last forty years because of its applications in the theory
of mechanical systems, such as robotic manipulators, satellites and
mobile robots. Research has largely focused on the study of models
and control strategies in the Lie group of rigid body motions $\SE 3$
and its subgroup $\SO 3$ of proper rotations \cite{BB:00,AD:98,BM:95,Selig:07}. 

Although it is usual to design attitude and rigid motion controllers
for mechanical systems respectively using rotation matrices and homogeneous
transformation matrices (HTM) \cite{BM:95}, it has been noted by
some authors that a non-singular representation, namely the unit quaternion
group $\unitquatgroup$ for rotations and the unit dual quaternions
$\unitdualquatgroup$ for rigid motions can bring computational advantages
\cite{Selig:07,Ador:11,WZ:14}. 

In scenarios where the state space of the dynamical system is not
the Euclidean space $\realset^{n}$ but a general differentiable manifold
$\mathcal{M}$\textemdash which is the case of $\SE 3$ and $\unitdualquatgroup$\textemdash some
difficulties in designing a stabilizing closed-loop controller may
arise. For instance, the topology of $\mathcal{M}$ may obstruct the
existence of a globally asymptotically stable equilibrium point in
any continuous vector field defined on $\mathcal{M}$: in \cite{BB:00}
it is proved that if $\mathcal{M}$ has the structure of a vector
bundle over a compact manifold $\mathcal{L}$, then no continuous
vector field on $\mathcal{M}$\textemdash indeed, nor in $\mathcal{L}$\textemdash has
a globally asymptotically stable equilibrium. In particular, this
means that it is impossible to design a continuous feedback that globally
stabilizes the pose of a rigid body, as in this case the closed-loop
system state space manifold is a trivial bundle over the compact manifold
$\SO 3$. The same topological obstruction is also present in the
group of unit dual quaternions since its underlying manifold is a
trivial bundle over the unit sphere $\unitquatset$ (see Theorem~\ref{thm:global_asymptotically_stable_DQ}
in Subsection~\ref{subsec:mathPrel:Unfeasible-Global-Stability}). 

On the other hand, due to the two-to-one covering map $\unitdualquatgroup\rightarrow\SE 3$,
the unit dual quaternion group is endowed with a double representation
for every pose in $\SE 3$. Neglecting the double covering yields
to the problem of unwinding whereby solutions close to the desired
pose in $\SE 3$ may travel farther to the antipodal unit dual quaternion
representing the same pose \cite{HWL:08}. There are few works on
unwinding avoidance in the context of pose stabilization using unit
dual quaternions \cite{HWL:08,HWL:08b,HWLS:08,WY:13}. All of them
are based on a discontinuous sign-based feedback approach. 

As shown by \cite{MST:11} for the particular case of $\unitquatgroup$,
the discontinuous sign-based approach may, however, be particularly
sensitive to measurement noises, and despite achieving global stability,
global attractivity properties may be detracted with arbitrarily small
measurement noises. As one would expect, the same happens with $\unitdualquatgroup$
and this will be shown in Theorem~\ref{thm:global_asymptotically_stable_DQ}.

As verified in Section~\ref{sec:Prior-work}, in $\unitdualquatgroup$
the lack of robustness is even more relevant as the discontinuity
of the controller not only affects the rotation, but may also disturb
and deteriorate the trajectory of the system translation. In this
context, even extremely small noises may lead to chattering, performance
degradation\textemdash and in the worst case, prevent stability. Summing
up, despite the solid contributions in the literature on dual quaternion
based controllers in the context of rigid body motion stabilization,
tracking, and multi-agent coordination \cite{HWL:08b,HWL:08,HWLS:08,WY:10,WY:11,WHYZ:12,WYL:12,WY:13},
control of manipulators and human-robot interaction \cite{Ador:10,PPAF:10,FAIB:13,2014_Figueredo_Adorno_Ishihara_Borges__IROS,2015_Adorno_Bo_Fraisse_ROB,2015_Marinho_Figueredo_Adorno__IROS},
and satellite and spacecraft tracking \cite{FT:13,FT:14,FKT:15},
it is important to emphasize that existing pose controllers are either
stable only locally (as we show in Section~\ref{sec:Mathematical-Preliminaries})
or have lack of robustness in the sense that they are sensitive to
arbitrarily small measurement noises (as we illustrate in Section~\ref{sec:Numerical-simulations}).
In other words, the topological constraints from $\unitdualquatgroup$,
most of them inherited from $\SE 3$, still pose a challenge\textemdash the
extension from results on attitude control to the problem of pose
control is not trivial\textemdash{} and there exists no result in
the literature that ensures \textbf{robust} global stability. 

\subsection*{\emph{Contributions}}

In this paper, a generalized robust hybrid control strategy for the
global stabilization of rigid motion kinematics within unit dual quaternions
framework is proposed. The strategy stems from the idea of the hybrid
kinematic control law with hysteresis switching proposed in \cite{MST:11}
to solve the non-robustness issue for quaternions. It is important
to emphasize that, albeit some algebraic identities in quaternion
algebra can be easily carried over to the dual quaternion algebra
by the principle of transference \cite{Chev:96,Selig:07}, the proposed
generalization does not follow by the principle of transference. In
fact, counterexamples shown in \cite{Chev:96} illustrate the failure
of the transfer principle outside the algebraic realm.

In summary, whereas unit quaternions are used to model only attitude
and perform a double cover for the Lie group $\SO 3$, unit dual quaternions
model the coupled attitude and position and perform a double cover
for the Lie group $\SE 3$. The necessity of different procedures
for quaternion and dual quaternion stems from their different topologies
and group structures. For example, the unit quaternion group is a
compact manifold, whereas the unit dual quaternion group is not a
compact manifold. This reflects in the use of distinct approaches
to controller design (see for instance \cite{BM:95}). It is also
interesting to highlight that due to $\SO 3$ being compact, it has
a natural bi-invariant metric, but the same can not be said from $\SE 3$
as it does not possess any bi-invariant metric. The unit dual quaternion
group is not a subgroup from $\unitquatgroup$\textemdash it is indeed
the other way around\textemdash and boundedness, geodesic distance,
norm properties, and other manifold features that are valid on $\unitquatset$
cannot be directly carried to $\unitdualquatgroup$. In this sense,
the extension of control results to $\unitdualquatgroup$ is not trivial,
which is reflected by the gap between quaternion based results and
dual quaternion based controllers\textemdash where the double covering
map is often neglected \cite{WY:10,WY:11,WYL:12,WHYZ:12,WZ:14}. To
overcome this context, we introduce a novel Lyapunov function that
exploits the algebraic constraints inherent from the unit dual quaternion
manifold, and a new robust hybrid stabilization controller for rigid
motion using $\unitdualquatgroup$ is derived.

\subsection*{\textit{Notation}}

Lowercase bold letters represent quaternions, such as $\quat q$.
Underlined lower case bold letters represent dual quaternions, such
as $\dq q$. The following notations will also be used:
\begin{flushleft}
\begin{tabular}{ll}
$\realset$ & set of real numbers;\tabularnewline
$\mathbb{R}_{\ge0}$ & set of non-negative real numbers;\tabularnewline
$\closedballset$ & closed unit ball in the Euclidean norm;\tabularnewline
$\quatset$ & set of quaternions;\tabularnewline
$\mathbb{H}_{0}$ & set of pure imaginary quaternions;\tabularnewline
$\dualquatset$ & set of dual quaternions;\tabularnewline
$\SO 3$ & $3$-dimensional Lie group of rotations;\tabularnewline
$\SE 3$ & $3$-dimensional Lie group of rigid body motions;\tabularnewline
$\unitquatgroup$ & Lie group of unit norm quaternions;\tabularnewline
$\unitdualquatgroup$ & Lie group of unit norm dual quaternions;\tabularnewline
$\unitquatset$ & underlying manifold of unit norm quaternions;\tabularnewline
$\unitdualquatset$ & underlying manifold of unit norm dual quaternions; \tabularnewline
$\mathcal{KL}$ & class of continuous functions $\beta:\mathbb{R}_{\ge0}\times\mathbb{R}_{\ge0}\rightarrow\mathbb{R}_{\ge0}$ \tabularnewline
 & such that for each fixed $s$, the function $\beta\left(r,s\right)$
is \tabularnewline
 & strictly increasing and $\beta\left(0,s\right)=0$ and for each fixed
$r$,\tabularnewline
 & the function $\beta\left(r,s\right)$ is decreasing and $\lim_{s\rightarrow\infty}\beta\left(r,s\right)=0$;\tabularnewline
$\|\cdot\|$ & Euclidean norm;\tabularnewline
$\quat u\cdot\quat v$ & dot product between pure imaginary quaternions $\quat u$ and $\quat v$;\tabularnewline
$\quat u\times\quat v$ & cross product between pure imaginary quaternions $\quat u$ and $\quat v$;\tabularnewline
$\overline{\mathrm{co}}(\cdot)$ & closure of the convex hull;\tabularnewline
$X+Y$ & Minkowski sum between the sets $X$ and $Y$;\tabularnewline
$x^{+}$ & denotes the next state of the hybrid system after a jump;\tabularnewline
\end{tabular}
\par\end{flushleft}

\section{Preliminaries \label{sec:Mathematical-Preliminaries}}

In this section, we provide for the reader basic concepts and a brief
theoretical background regarding quaternions and dual quaternion representation
for rigid body motion. We also present the topological constraints\textemdash which
affect any mathematical structure that represents rigid motion\textemdash imposed
by dual quaternions.

\subsection{Quaternions \label{subsec:Quaternions}}

The quaternion algebra is a four-dimensional associative division
algebra over $\realset$ introduced by Hamilton \cite{Hami:1844}
to algebraically express rotations in the three-dimensional space.
The elements $1$,$\imi,\imj,\imk$ are the basis of this algebra,
satisfying 
\[
\imi^{2}=\imj^{2}=\imk^{2}=\imi\imj\imk=-1,
\]
and the set of quaternions is defined as
\[
\smash{\mathbb{H}\triangleq\biggl\{\eta+\text{\ensuremath{\imi}}\mu_{1}+\imj\mu_{2}+\imk\mu_{3}:\eta,\mu_{1},\mu_{2},\mu_{3}\in\mathbb{R}\biggr\}.}
\]
 Consider a quaternion $\quat q=\eta+\text{\ensuremath{\imi}}\mu_{1}+\imj\mu_{2}+\imk\mu_{3}$;
for ease of notation, it may be denoted as
\[
\smash{\quat q=\eta+\quat{\mu},\quad\text{with }\quad\quat{\mu}=\text{\ensuremath{\imi}}\mu_{1}+\imj\mu_{2}+\imk\mu_{3}.}
\]
In addition, it may be decomposed into a real component and an imaginary
component: $\real{\quat q}\triangleq\eta$ and $\imag{\quat q}\triangleq\quat{\mu}$,
such that $\quat q=\real{\quat q}+\imag{\quat q}$. The quaternion
conjugate is given by $\quat q^{*}\triangleq\real{\quat q}-\imag{\quat q}$.
Pure imaginary quaternions are given by the set 
\[
\mathbb{H}_{0}\triangleq\left\{ \quat q\in\mathbb{H}:\;\real{\quat q}=0\right\} 
\]
and are very convenient to represent vectors of $\mathbb{R}^{3}$
within the quaternion formalism by means of a trivial isomorphism,
which implies $\mathbb{H}_{0}\cong\mathbb{R}^{3}$. Both cross product
and dot product are defined for elements of $\mathbb{H}_{0}$ and
they are analogous to their counterparts in $\mathbb{R}^{3}$. More
specifically, given $\quat u,\quat v\in\mathbb{H}_{0}$, the dot product
is defined as 
\begin{align*}
\quat u\cdot\quat v & \ \smash{\triangleq\ }\smash{-\frac{\quat u\quat v+\quat v\quat u}{2}},
\end{align*}
and the cross product is given by
\[
\quat u\times\quat v\ \smash{\triangleq\ }\smash{\frac{\quat u\quat v-\quat v\quat u}{2}}.
\]
 The quaternion norm is defined as $\norm{\quat q}{\triangleq}\sqrt{\quat q\quat q^{*}}.$
Unit quaternions are defined as the quaternions that lie in the subset
\[
\unitquatset\triangleq\left\{ \quat q\in\quatset:\;\norm{\quat q}=1\right\} .
\]

The set $\unitquatset$ forms, under multiplication, the Lie group
$\unitquatgroup$, whose identity element is $1$ and group inverse
is given by the quaternion conjugate $\quat q^{\ast}$.

Analogously to the way complex numbers are used to represent rotations
in the plane, unit quaternions represent rotations in the three-dimensional
space. Indeed, an arbitrary rotation $\theta$ around an axis $\myvec n=\imi n_{x}+\imj n_{y}+\imk n_{z}$
is represented by the unit quaternion $\quat r=\cos\left(\theta/2\right)+\sin\left(\theta/2\right)\myvec n$.
Furthermore, since the unit quaternion group double covers the rotation
group $\SO 3$, the unit quaternion $-\quat r$ also represents the
same rotation associated to $\quat r$ \cite{Selig:07}. 

\subsection{Dual Quaternions}

Similarly to how the quaternion algebra was introduced to address
rotations in the three-dimensional space, the dual quaternion algebra
was introduced by Clifford and Study \cite{Clif:1873,Stud:1891} to
describe rigid body movements. This algebra is constituted by the
set
\[
\smash{\dualquatset\triangleq\left\{ \quat q+\dual\quat q':\;\quat q,\quat q'\in\quatset\right\} ,}
\]
where $\dual$ is called the dual unit and is nilpotent\textemdash that
is, $\dual{\not=}0$, but $\dual^{2}{=}0$. Given $\dq q=\eta+\quat{\mu}+\dual\bigl(\eta'+\quat{\mu}'\bigr),$
we define $\real{\dq q}\triangleq\eta+\dual\eta'$ and $\imag{\dq q}\triangleq\quat{\mu}+\dual\quat{\mu}'$,
such that $\dq q=\real{\dq q}+\imag{\dq q}$. The dual quaternion
conjugate is $\dq q^{*}\triangleq\real{\dq q}-\imag{\dq q}$.

Under dual quaternion multiplication, the subset of dual quaternions
\begin{equation}
\smash{\unitdualquatset\triangleq\left\{ \quat q+\dual\quat q'\in\dualquatset:\;\norm{\quat q}=1,\,\quat q\quat q'^{*}+\quat q'\quat q^{*}=0\right\} ,}\label{eq:DQ_Group_Description}
\end{equation}
forms a Lie group \cite{WHYZ:12} called unit dual quaternions group
$\unitdualquatgroup$, whose identity is $1$ and group inverse is
the dual quaternion conjugate. 

An arbitrary rigid displacement characterized by a rotation $\quat r\in\unitquatgroup$,
with $\quat r=\cos\left(\theta/2\right)+\sin\left(\theta/2\right)\myvec n$,
followed by a translation $\quat p\in\mathbb{H}_{0}$, with $\quat p=p_{x}\imi+p_{y}\imj+p_{z}\imk$,
is represented by the unit dual quaternion \cite{Ador:11,HWLS:08}\footnote{Similarly, the rigid motion could also be represented by a translation
$\overline{\quat p}$ followed by a rotation $\quat r$ \cite{WHHLL:05}
resulting in the dual quaternion $\dq q=\quat r+(1/2)\dual\overline{\quat p}\quat r$.
Both $\overline{\quat p}$ and $\quat p$ are related by $\overline{\quat p}=\quat r\quat p\quat r^{*}$.}
\[
\smash{\dq q=\quat r+\dual\frac{1}{2}\quat r\quat p.}
\]
 The unit dual quaternions group double covers $\SE 3$ and thus any
displacement $\dq q$ can also be described by $-\dq q$.

\subsection{Description of rigid motion and topological constraints\label{subsec:mathPrel:Unfeasible-Global-Stability}}

Since unit quaternions describe the attitude of a rigid body, they
are used to represent a rotation between the body frame and the inertial
frame. In this sense, the kinematic equation of a rotation represented
by the unit quaternion $\quat q$ is expressed as
\begin{equation}
\dot{\quat q}=\frac{1}{2}\quat q\quat{\omega},\label{eq:Q_kinematics}
\end{equation}
where $\quat{\omega}\in\mathbb{H}_{0}$ is the angular velocity given
in the body frame \cite{Chou:92}. 

Similarly, the unit dual quaternion $\dq q$ describes the coupled
attitude and position. The first order kinematic equation of a rigid
body motion in the inertial frame is given by
\begin{equation}
\dot{\dualquaternion q}=\frac{1}{2}\dualquaternion q\dq{\omega},\label{eq:DQ_kinematics}
\end{equation}
where $\dq{\omega}$ is the twist in body frame given by
\begin{equation}
\dq{\omega}=\quat{\omega}+\dual\left[\dot{\quat p}+\quat{\omega}\times\quat p\right].\label{eq:generalized_spatial_velocity}
\end{equation}

The remarkable similarity between equations (\ref{eq:Q_kinematics})
and (\ref{eq:DQ_kinematics}) is due to the principle of transference,
whose various forms as stated in \cite{Chev:96} can be summarized
in modern terms as \cite[Sec 7.6]{Selig:07}: ``All representations
of the group $\SO 3$ become representations of $\SE 3$ when tensored
with the dual numbers.'' This means that several properties and algebraic
identities of $\SO 3$ and the quaternions can be carried to $\SE 3$
and the dual quaternions algebra, respectively. 

The principle of transference may mislead one to think that every
theorem in quaternions can be transformed to another theorem in dual
quaternions by a transference process. However, this is not the case,
as shown by counterexamples in \cite{Chev:96}. Therefore, properties
and phenomena related to quaternion motions like topological obstructions
and unwinding may not follow by direct use of transference and have
to be verified for dual quaternions. 

We first verify that $\text{Spin}(3)\ltimes\mathbb{R}^{3}$ presents
an obstruction for the global asymptotic stability of a continuous
vector field on its underlying manifold.
\begin{thm}
\label{thm:global_asymptotically_stable_DQ} Let $f$ be a continuous
vector field defined on the underlying manifold $\unitdualquatset$
of the Lie group of unit dual quaternions. Then there exists no equilibrium
point of $f$ that is globally asymptotically stable. 
\end{thm}
\begin{proof}
For an arbitrary unit dual quaternion element $\dq q\in\dq S$, with
$\dq q=\quat q+\dual\quat q'=\eta+\quat{\mu}+\dual\left(\eta'+\quat{\mu}'\right)$,
as defined in (\ref{eq:DQ_Group_Description}), it is possible to
verify by direct calculation that the constraint \textbf{$\quat q\quat q'^{*}+\quat q'\quat q^{*}=0$}
implies 
\begin{equation}
\eta\eta'+\quat{\mu}\cdot\quat{\mu}'=0.\label{eq:orthonormal_spaces}
\end{equation}
Furthermore, since $\norm{\quat q}=1$, then $\quat q$ lies in $\unitquatset$.
In addition, $\mathbb{H}$ is isomorphic to $\mathbb{R}^{4}$ as a
vector space, which implies that $\quat q'\in\mathbb{H}$ lies in
a three-dimensional hyperplane, with $\quat q$ being its normal vector,
due to constraint (\ref{eq:orthonormal_spaces}). In this sense, $\unitdualquatset$
can be regarded as the product manifold $\unitquatset\times\mathbb{R}^{3}$
\cite{1990_Mccarthy_BOOK}. 

The product $\unitquatset\times\mathbb{R}^{3}$ equipped with the
projection $\unitquatset\times\mathbb{R}^{3}\rightarrow\unitquatset$
given by $\dq q\mapsto\quat q$ yields a vector bundle $\unitquatset\times\mathbb{R}^{3}$
onto $\unitquatset$, namely the trivial bundle \cite{Lee:13}. Since
$\unitquatset$ is compact, it follows from Theorem~1 of \cite{BB:00}
(for the reader's convenience, this theorem is reproduced in Theorem~\ref{thm:BB:00}
of the Appendix) that there is no equilibrium point of any continuous
vector field $f$ that is globally asymptotically stable.
\end{proof}

\section{Prior work on pose stabilization\label{sec:Prior-work}}

Due to the topological constraint described in Theorem~\ref{thm:global_asymptotically_stable_DQ},
there is no continuous state feedback controller on $\unitdualquatset$
that can globally asymptotically stabilize (\ref{eq:DQ_kinematics})
to a rest configuration. Indeed, the two-to-one covering map from
$\unitdualquatgroup$ to $\SE 3$ renders a closed-loop system with
two distinct equilibria $\dq{q_{e}}$ and $-\dq{q_{e}}$. 

Since both $\pm\dq{q_{e}}$ correspond to the same configuration in
$\SE 3$, solutions neglecting the double cover (see, for example,
\cite{PPAF:10,WY:10,WY:11,WHYZ:12,WYL:12,FAIB:13}) are expected to
exhibit the unwinding phenomenon \cite{BB:00}, that is, solutions
starting arbitrarily close to the desired pose in $\SE 3$\textemdash represented
by both stable and unstable points in $\unitdualquatgroup$\textemdash may
travel to the farther stable point instead to the nearest unstable
point (see, for example, Fig.$\,$\ref{fig:Unwinding}). The sole
contributions in the sense of avoiding the unwinding and stabilizing~(\ref{eq:DQ_kinematics})
to the set $\{\pm1\}$ are based on a pure discontinuous control law
introduced in \cite{HWL:08b,HWL:08,HWLS:08}. In terms of the components
of $\dq q=\eta+\quat{\mu}+\dual\bigl(\eta'+\quat{\mu}'\bigr)$, this
discontinuous control law is given by\footnote{The discontinuous kinematic control law in \cite{HWL:08b,HWL:08,HWLS:08}
contains a typo that has been fixed in \cite{WY:13}. Different from
(\ref{eq:Han_controller}), in \cite{HWL:08b,HWL:08,HWLS:08,WY:13}
the controller is expressed in terms of the logarithm of a unit dual
quaternion.}
\begin{equation}
\dq{\omega}{=}\begin{cases}
-2k\left[\mathrm{acos}(\eta)\frac{\quat{\mu}}{\norm{\quat{\mu}}}{+}\dual\quat v\right], & \text{ if \ensuremath{\eta{\ge}0}},\\
-2k\left[\left(\mathrm{acos}(\eta)-\pi\right)\frac{\quat{\mu}}{\norm{\quat{\mu}}}{+}\dual\quat v\right], & \text{\text{ if \ensuremath{\eta{<}0}}},
\end{cases}\label{eq:Han_controller}
\end{equation}
where $\quat v=\eta\quat{\mu}'{-}\eta'\quat{\mu}{-}\quat{\mu}\times\quat{\mu}'$
and $k$ is a proportional gain.

Albeit this control law avoids unwinding, a careful look reveals a
strong sensitivity around attitudes that are up to $\pi$ away from
the desired attitude about some axis\textemdash that is, $\eta=0$.
In view of Theorem~2.6 of \cite{SMET:06}, one can see that such
control law isn't robust in the sense that arbitrarily small measurement
noises can force $\eta$ to stay near to $0$ for initial conditions
within its neighborhood. Indeed, similar to Theorem~3.2 of \cite{MST:11},
one can even exhibit an explicit noise signal to persistently trap
the system about a fixed pose, thus preventing its stability. To illustrate
the sensitivity of pure discontinuous state feedback controllers,
we introduce a simple case study in which the trajectory of (\ref{eq:DQ_kinematics})
is simulated using the discontinuous control law (\ref{eq:Han_controller})
in the presence of a random measurement noise\footnote{The simulation has been performed in accordance with the procedures
described in Section~\ref{sec:Numerical-simulations}.}\textemdash the results are shown in Fig.$\,$\ref{fig:Pure Discontinuous Controller - Chattering - raster version}.
The trajectory of the closed-loop system exhibits chattering in the
neighborhood of the discontinuity\textemdash that lies in $\eta=0$\textemdash as
a result of the measurement noise. The performance degradation stems
from infinitely fast switches in (\ref{eq:Han_controller}). Furthermore,
the chattering influence over the system is not restricted to the
attitude error and may also impact on the resulting trajectory of
the translation, as shown in Fig.$\,$\ref{fig:Pure Discontinuous Controller - Chattering - raster version}(b).
In this sense, the lack of robustness of a discontinuous solution
may lead to chattering in orientation and to additional disturbances
in the translation of the rigid motion in the presence of arbitrarily
small random noises. 

\begin{figure}
\begin{centering}
\includegraphics[width=0.93\columnwidth]{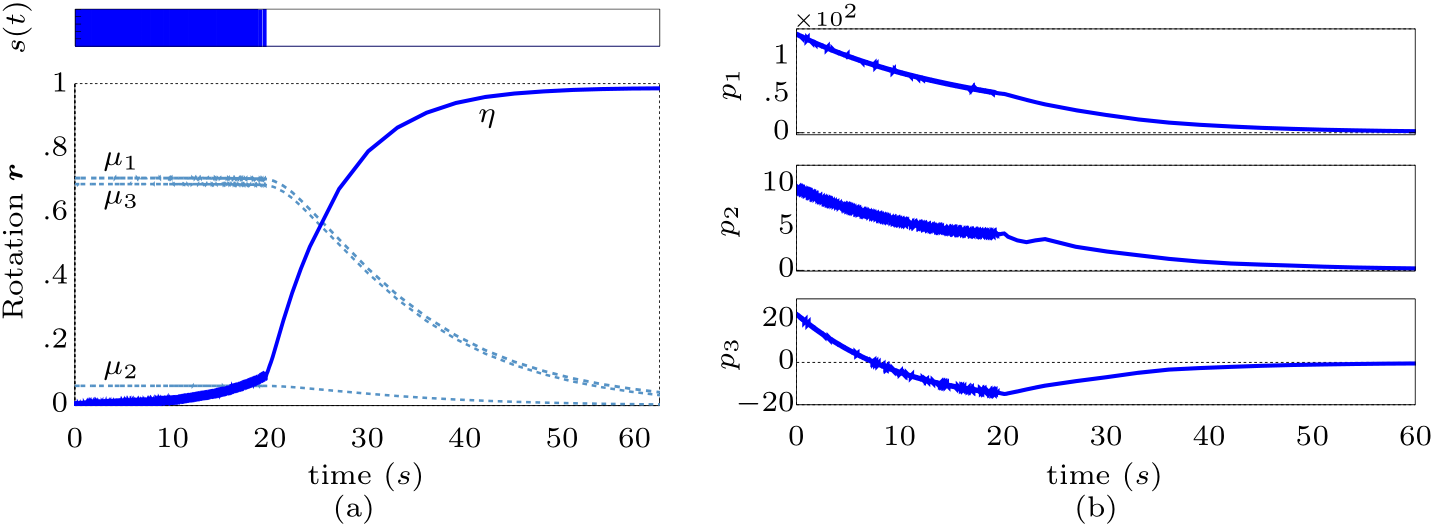}
\par\end{centering}
\caption{(\emph{a}) Trajectory of the rotation unit quaternion $\protect\quat r$
in terms of $\eta$ and $\protect\myvec{\mu}$ (dashed line) with
switches along time between the discontinuous control laws in (\ref{eq:Han_controller})
represented by $s(t)$. (\emph{b}) Trajectory of the three-dimensional
translation elements $\protect\myvec p=p_{1}\protect\imi+p_{2}\protect\imj+p_{3}\protect\imk$.
\label{fig:Pure Discontinuous Controller - Chattering - raster version}}
\end{figure}

\section{Kinematic hybrid control law for robust global pose stability\label{sec:Stability}}

In this section, we address the control design problem for globally
stabilizing a rigid body coupled rotational and translational kinematics
with no representation singularities. The proposed solution copes
with the topological constraint inherent from the $\unitdualquatgroup$
parametrization while also ensuring robustness against measurement
noises. 

To avoid the unwinding phenomenon and the lack of robustness from
pure discontinuous solutions, we appeal to the hybrid system formalism
of \cite{GST:09}. This formalism combines both continuous-time and
discrete-time dynamics and is useful to formally analyze hysteresis-based
control laws, such as the proposed solution. A hybrid system is given
by the constrained differential inclusions
\begin{gather}
\begin{split}\dot{x}\in F\left(x\right), & \quad x\in C,\\
x^{+}\in G\left(x\right), & \quad x\in D,
\end{split}
\label{eq:hybrid-systems-notation}
\end{gather}
where $x^{+}$ denotes the next state of the hybrid system after a
jump. The flow map $F$ and the jump map $G$ are set-valued functions,
respectively, modeling the continuous and the discrete time dynamics
of the system. The flow set $C$ and the jump set $D$ are the respective
sets where the evolution occur. The following concepts of set-valued
analysis will also be used: a set-valued mapping $F$ is \textit{\emph{outer}}\textit{
}\textit{\emph{semicontinuous}} if its graph is closed and $F$ is
\textit{\emph{locally}}\textit{ }\textit{\emph{bounded}}\textemdash that
is, if for any compact set $K$, there exists $m>0$ such that $F\left(K\right)\subset m\closedballset$,
where $\closedballset$ is the closed unit ball in the Euclidean space
of the convenient dimension \cite{RW:98}. For more details, the reader
is referred to \cite{GST:09,GST:12}.

To solve the problem of robust global asymptotic stabilization of
(\ref{eq:DQ_kinematics}), we propose a generalization to the hysteresis-based
hybrid control law of \cite{MST:11} that extends the attitude stabilization
to render both coupled kinematics\textemdash attitude and translation\textemdash stable.
The proposed control law is defined as
\begin{equation}
\dq{\omega}\triangleq-k_{1}h\quat{\mu}-\dual k_{2}\eta\quat{\mu}',\label{eq:hybrid control law - twist}
\end{equation}
 where $k_{1},k_{2}\in\realset_{+}^{*}$ are control gains and $h\in\{-1,1\}$
is a memory state with hysteresis characterized by a parameter $\delta\in\left(0,1\right)$.
The memory state $h$ has its dynamics defined by
\begin{gather}
\begin{split}\dot{h}\triangleq0, & \quad\text{when }\left(\dualquaternion q,h\right)\text{ are such that }\ensuremath{h\eta\ge-\delta,}\\
h^{+}\in\overline{\sgn}\left(\eta\right), & \quad\text{when }\left(\dualquaternion q,h\right)\text{ are such that }h\eta\leq-\delta,
\end{split}
\label{eq:hysteresis_parameter_def}
\end{gather}
where $\overline{\sgn}$ is the set-valued function defined as
\[
\overline{\sgn}\left(s\right)\triangleq\begin{cases}
\left\{ 1\right\} , & s>0,\\
\left\{ -1\right\} , & s<0,\\
\left\{ -1,1\right\} , & s=0.
\end{cases}
\]
In terms of the hybrid formalism (\ref{eq:hybrid-systems-notation}),
the closed loop system made by (\ref{eq:DQ_kinematics}), (\ref{eq:hybrid control law - twist})
and (\ref{eq:hysteresis_parameter_def}) is characterized as
\begin{flalign}
 & \begin{aligned}F\left(\dualquaternion q,h\right)=\left(\frac{1}{2}\dualquaternion q\dq{\omega},0\right),C=\left\{ \left(\dualquaternion q,h\right)\in\left(\unitdualquatgroup\right)\times\{-1,1\}:\ h\eta\ge-\delta\right\} ,\\
G\left(\dualquaternion q,h\right)\in\left(\dualquaternion q,\overline{\sgn}\left(\eta\right)\right),D=\left\{ \left(\dualquaternion q,h\right)\in\left(\unitdualquatgroup\right)\times\{-1,1\}:\ h\eta\le-\delta\right\} ,
\end{aligned}
\label{eq:hybrid sets}
\end{flalign}
where $\dq{\omega}$ is defined as in (\ref{eq:hybrid control law - twist})
and $h$ as in (\ref{eq:hysteresis_parameter_def}).

Consider $\quat{\mu}=\imi\mu_{1}+\imj\mu_{2}+\imk\mu_{3}$ and $\quat{\mu'}=\imi\mu'_{1}+\imj\mu'_{2}+\imk\mu'_{3}$.
The map from $\dualquatset$ to $\mathbb{R}^{8}$ given by
\begin{equation}
\dq q=\eta+\myvec{\mu}+\dual\left(\eta'+\myvec{\mu}'\right)\mapsto\vector\left(\dq q\right)=[\eta,\mu_{1},\mu_{2},\mu_{3},\eta',\mu'_{1},\mu'_{2},\mu'_{3}]^{T}\label{eq:vec_map_def}
\end{equation}
is a vector space isomorphism whose inverse will be denoted by $\underline{\vector}$.

The Hamilton operator \cite{1990_Mccarthy_BOOK,Ador:11} provides
a matrix representation for the algebraic multiplication through the
map $\hami +:\dualquatset\rightarrow\mathbb{R}^{8\times8}$ satisfying
\begin{equation}
\vector\left(\dq q_{1}\dq q_{2}\right)=\hamilton +{\dq q_{1}}\vector\left(\dq q_{2}\right)\label{eq:hamilton_operator}
\end{equation}
for any $\dq q_{1},\dq q_{2}\in\dualquatset$. Explicitly, the Hamilton
operator is given by
\[
\dq q=\quat q+\dual\quat q'\mapsto\hamilton +{\dq q}=\begin{bmatrix}\hami +_{4}(\quat q) & 0_{4}\\
\hami +_{4}(\quat q') & \hami +_{4}(\quat q)
\end{bmatrix},
\]
where $\hami +_{4}:\quatset\rightarrow\mathbb{R}^{4\times4}$ is the
map
\[
\quat q=\eta+\imi\mu_{1}+\imj\mu_{2}+\imk\mu_{3}\mapsto\hami +_{4}(\quat q)=\begin{bmatrix}\eta & -\mu_{1} & -\mu_{2} & -\mu_{3}\\
\mu_{1} & \eta & -\mu_{3} & \mu_{2}\\
\mu_{2} & \mu_{3} & \eta & -\mu_{1}\\
\mu_{3} & -\mu_{2} & \mu_{1} & \eta
\end{bmatrix}.
\]

Let $\myvec x=(x_{1},\ldots,x_{8})\in\mathbb{R}^{8}$ and $y\in\mathbb{R}$.
Based on (\ref{eq:vec_map_def}) and (\ref{eq:hamilton_operator}),
the maps $F$ and $G$ of (\ref{eq:hybrid sets}) induce the function
$\vec{F}:\mathbb{R}^{9}\rightarrow\mathbb{R}^{9}$ and the set-valued
mapping $\vec{G}:\mathbb{R}^{9}\rightrightarrows\mathbb{R}^{9}$ given
by 
\begin{equation}
\vec{F}(\myvec x,y)=\left(\frac{1}{2}\hamilton +{\underline{\vector}(\myvec x)}\vector\left(\dq{\omega}\right),0\right),\ \vec{G}(\myvec x,y)\in\left(\myvec x,\overline{\sgn}\left(x_{1}\right)\right),\label{eq:hybrid maps vec}
\end{equation}
where
\[
\vector\left(\dq{\omega}\right)=[0,-k_{1}hx_{2},-k_{1}hx_{3},-k_{1}hx_{4},0,-k_{2}x_{1}x_{6},-k_{2}x_{1}x_{7},-k_{2}x_{1}x_{8}]^{T}.
\]
Similarly, the sets $C$ and $D$ of (\ref{eq:hybrid sets}) induce
the subsets $\vec{C}$ and $\vec{D}$ of $\mathbb{R}^{9}$ given by
\begin{gather}
\begin{split}\vec{C}=\left\{ \left(\myvec x,y\right)\in\mathbb{R}^{8}\times\mathbb{R}:\left(\myvec x,y\right)\in\unitdualquatset\times\{-1,1\}\text{ and }yx_{1}\ge-\delta\right\} ,\\
\vec{D}=\left\{ \left(\myvec x,y\right)\in\mathbb{R}^{8}\times\mathbb{R}:\left(\myvec x,y\right)\in\unitdualquatset\times\{-1,1\}\text{ and }yx_{1}\le-\delta\right\} .
\end{split}
\label{eq:hybrid sets vec}
\end{gather}

The following lemma proves that the hybrid system induced by (\ref{eq:DQ_kinematics}),
(\ref{eq:hybrid control law - twist}) and (\ref{eq:hysteresis_parameter_def})
satisfies some properties which helps to prove the stability of the
system and its robustness.
\begin{lem}
\label{lem:basic-assumptions} The maps $\vec{F}$ and $\vec{G}$
defined on (\ref{eq:hybrid maps vec}) and the sets $\vec{C}$ and
$\vec{D}$ defined on (\ref{eq:hybrid sets vec}) satisfy the following
properties:

\begin{enumerate}
\item $\vec{C}$ and $\vec{D}$ are closed sets in $\mathbb{R}^{9}$.
\item $\vec{F}:\mathbb{R}^{9}\rightarrow\mathbb{R}^{9}$ is continuous.
\item $\vec{G}:\mathbb{R}^{9}\rightrightarrows\mathbb{R}^{9}$ is an outer
semicontinuous set-valued mapping, locally bounded and $\vec{G}\left(\myvec x,h\right)$
is nonempty for each $\left(\myvec x,h\right)\in\vec{D}$.
\end{enumerate}
\end{lem}
\begin{proof}
The proof is based on Lemma~5.1 of \cite{MST:11}. Setting $\delta\in(0,1)$,
consider the continuous map $\tau:\mathbb{R}^{9}\rightarrow\mathbb{R}$
given by $\tau\left(x_{1},\ldots,x_{8},y\right)=yx_{1}+\delta$. The
restriction $\tau|_{\unitdualquatset\times\{-1,1\}}:\unitdualquatset\times\{-1,1\}\rightarrow\mathbb{R}$
of this map to $\unitdualquatset\times\{-1,1\}$ is also continuous
\cite[Theorem 8]{Base:13}. Moreover, by the definition of the sets
$\vec{C}$ and $\vec{D}$, we have that
\begin{gather*}
\vec{C}=\tau|_{\unitdualquatset\times\{-1,1\}}^{-1}\left([0,+\infty)\right),\\
\vec{D}=\tau|_{\unitdualquatset\times\{-1,1\}}^{-1}\left((-\infty,0]\right).
\end{gather*}
Since the preimage of a closed set under a continuous mapping is closed,
$\vec{C}$ and $\vec{D}$ are closed in $\unitdualquatset\times\{-1,1\}$.
We also have that $\unitdualquatset\times\{-1,1\}$ is closed in $\mathbb{R}^{9}$.
In fact, consider the continuous functions $p,\,d:\mathbb{R}^{8}\rightarrow\mathbb{R}$
given respectively by
\begin{eqnarray*}
p(\eta,\mu_{1},\mu_{2},\mu_{3},\eta',\mu'_{1},\mu'_{2},\mu'_{3}) & = & [\eta,\mu_{1},\mu_{2},\mu_{3}][\eta,\mu_{1},\mu_{2},\mu_{3}]^{T}-1,\\
d(\eta,\mu_{1},\mu_{2},\mu_{3},\eta',\mu'_{1},\mu'_{2},\mu'_{3}) & = & [\eta,\mu_{1},\mu_{2},\mu_{3}][\eta',\mu'_{1},\mu'_{2},\mu'_{3}]^{T}.
\end{eqnarray*}
By the definition of $p$ and $d$, $\unitdualquatset=p^{-1}(\{0\})\cap d^{-1}(\{0\})$.
Since $\{0\}$ is a closed set of $\mathbb{R}$, the sets $p^{-1}(\{0\})$
and $d^{-1}(\{0\})$ are closed and their intersections are closed.
Thus, $\unitdualquatset$ is closed in $\mathbb{R}^{8}$. Moreover,
the set $\{-1,1\}$ is closed in $\mathbb{R}$, therefore the Cartesian
product $\unitdualquatset\times\{-1,1\}$ is closed in $\mathbb{R}^{9}$.

Since $\unitdualquatset\times\{-1,1\}$ is closed in $\mathbb{R}^{9}$,
$\vec{C}$ and $\vec{D}$ are also closed in $\mathbb{R}^{9}$. On
the account that each component of $\vec{F}$ is a polynomial, the
map $\vec{F}$ is continuous.

The graph of the set-valued mapping $\vec{G}$ is given by $\left\{ (\myvec x,y,z):z\in\vec{G}\left(\myvec x,y\right)\right\} =\mathbb{R}^{8}\times\mathbb{R}\times\mathbb{R}^{8}\times\{-1,1\}$.
Since this set is closed, it follows by definition that $\vec{G}$
is outer semicontinuous.\footnote{The graph of a set-valued mapping $F:X\rightrightarrows Y$ is defined
by $\left\{ (x,y)\in X\times Y:x\in X,\,y\in F(x)\right\} $. $F$
is outer semicontinuous if its graph is a closed set of $X\times Y$
\cite{MST:11}.}

Furthermore, $\vec{G}$ is locally bounded because given any compact
set $K$$\subset\mathbb{R}^{9}$, $\vec{G}(K)\subset K\times\{-1,1\}$
and thus $\vec{G}(K)$ is bounded. Finally, by the definition of $\vec{G}$,
$\vec{G}(\myvec x,y)$ is nonempty for every $(\myvec x,y)\in\vec{D}$. 
\end{proof}
\begin{thm}
With $\dq{\omega}$ defined as in (\ref{eq:hybrid control law - twist}),
the equilibrium points of the closed loop system made by (\ref{eq:DQ_kinematics}),
(\ref{eq:hybrid control law - twist}) and (\ref{eq:hysteresis_parameter_def})
are $\pm1$ and the set $\left\{ \pm1\right\} $ is asymptotically
stable.
\end{thm}
\begin{proof}
Using the control law (\ref{eq:hybrid control law - twist})-(\ref{eq:hysteresis_parameter_def})
in (\ref{eq:DQ_kinematics}), the closed-loop system is
\begin{gather}
\begin{split}\dot{\dq q} & =\dot{\eta}+\dot{\quat{\mu}}+\dual\left(\dot{\eta}'+\dot{\quat{\mu}}'\right),\:\text{ with }\dot{\eta}=\frac{1}{2}k_{1}h\norm{\quat{\mu}}^{2},\\
\dot{\quat{\mu}} & =-\frac{1}{2}\eta k_{1}h\quat{\mu},\hspace{1cm}\dot{\eta}'=\frac{1}{2}\left(k_{1}h+k_{2}\eta\right)\quat{\mu}'\cdot\quat{\mu},\\
\dot{\quat{\mu}}' & =\frac{1}{2}\left[\left(k_{1}h-k_{2}\eta\right)\quat{\mu}\times\quat{\mu}'-k_{1}h\eta'\quat{\mu}-k_{2}\eta^{2}\quat{\mu}'\right].
\end{split}
\label{eq:closed_loop_with_hybrid_control_law}
\end{gather}

To find the equilibria of (\ref{eq:closed_loop_with_hybrid_control_law}),
note that $\dot{\dualquaternion q}=0$ implies $\quat{\mu}=0.$ From
the unit sphere constraint~(\ref{eq:DQ_Group_Description}), it also
follows that $\eta=\pm1$ whereby we can find that $\quat{\mu}'=0$.
In this context, the constraint (\ref{eq:orthonormal_spaces}) also
renders $\eta'=0$. Hence, the set of equilibrium points of (\ref{eq:closed_loop_with_hybrid_control_law})
is the set $\left\{ \pm1\right\} $.

To study the stability of the set of equilibrium points $\left\{ \pm1\right\} $,
let us regard the Lyapunov candidate function
\begin{equation}
V(\dq q,h)=2\left(1-h\eta\right)+\eta'^{2}+\norm{\quat{\mu}'}^{2}.\label{eq:Lyapunov_Definition}
\end{equation}

Since $\eta\in\left[-1,1\right]$ and $h\in\left\{ -1,1\right\} $,
one has that $\left(1-h\eta\right)\ge0.$ Therefore, $V$ is a positive
semidefinite function. The condition $V=0$ implies $0\leq2\left(1-h\eta\right)=-\eta'^{2}-\norm{\quat{\mu}'}^{2}\le0$
which yields $\eta'=0$, $\quat{\mu}'=0$ and $h\eta=1$, that is,
$\dualquaternion q=\pm1$. Hence, $V$ is a positive definite function.

Taking the time-derivative of $V$ and using (\ref{eq:closed_loop_with_hybrid_control_law})
yields
\begin{align*}
\dot{V} & =-2h\dot{\eta}+2\eta'\dot{\eta}'+2\quat{\mu}'\cdot\dot{\quat{\mu}}'\\
 & =-h^{2}k_{1}\norm{\quat{\mu}}^{2}-\eta^{2}\eta'^{2}k_{2}-\eta^{2}\norm{\quat{\mu}'}^{2}k_{2}\le0.
\end{align*}
In addition, $\dot{V}=0$ if and only if $\dualquaternion q\in\{\pm1\}$.
Moreover, $V$ also decreases over jumps of the closed loop system
since for $h\eta<-\delta<0$ one has that
\[
V(\dq q,h^{+})-V(\dq q,h)=4h\eta<0.
\]

Thus, asymptotically stability of the set $\left\{ \pm1\right\} $
follows from Lemma~\ref{lem:basic-assumptions} and by Theorem~20
of \cite{GST:09}. It is also important to highlight that the closed-loop
differential equation is well-posed \cite[Prop. 2.1]{Sach:09} as
$\dq{\omega}$ is in the Lie algebra of $\unitdualquatgroup$.
\end{proof}
\begin{rem}
At a first glance, one could imagine that due to the transference
principle \cite{Chev:96}, the extension of rotation stabilizers (e.g.,
the ones of \cite{BMS:95,MST:11}) to full rigid body stabilizers
would be trivial, only requiring the substitution of adequate variables
as in (\ref{eq:Q_kinematics}) and (\ref{eq:DQ_kinematics}). However,
for stability analysis based on Lyapunov functions, this supposition
doesn't even make sense, since a Lyapunov function is a real-valued
function and never a dual-number valued function. As a consequence,
stabilization in $\unitdualquatgroup$ using dual quaternions required
one independent study from the quaternion stabilization analysis in
$\unitquatgroup$. The necessity of different procedures for quaternion
and dual quaternion is also inferred by remembering that due to the
fact that $\SO 3$ is compact and $\SE 3$ is not, it was required
one controller design procedure for each case in \cite{BM:95}. 
\end{rem}
Similarly to the rotation controllers proposed in \cite{MST:11},
the proposed pose controller doesn't exhibit Zeno behavior \cite{GST:09}.
This is shown in the next lemma. 
\begin{lem}
For any compact set $K\subset\unitdualquatset\times\left\{ -1,1\right\} $,
if $x$ is a solution of (\ref{eq:DQ_kinematics}), (\ref{eq:hybrid control law - twist})
and (\ref{eq:hysteresis_parameter_def}) with initial state in $K$,
then the number of jumps is bounded. 
\end{lem}
\begin{proof}
Similar to Theorem~5.3 of \cite{MST:11}. 
\end{proof}
The stability robustness will be characterized by the system's resistance
against $\alpha$-perturbations: given $\alpha>0$, the $\alpha$-perturbation
of the hybrid system given by $\vec{F},\vec{G}$ as in (\ref{eq:hybrid maps vec}),
and $\vec{C},\vec{D}$ as in (\ref{eq:hybrid sets vec}), is given
by 
\begin{gather*}
\vec{C}_{\alpha}\triangleq\left\{ x\in\mathbb{R}^{9}:\left(x+\alpha\closedballset\right)\cap\vec{C}\ne\emptyset\right\} ,\\
\vec{F}_{\alpha}\left(x\right)\triangleq\overline{\mathrm{co}}\,\vec{F}\left(\left(x+\alpha\closedballset\right)\cap\vec{C}\right)+\alpha\closedballset,\ \text{for all }x\in\vec{C}_{\alpha},\\
\vec{D}_{\alpha}\triangleq\left\{ x\in\mathbb{R}^{9}:\left(x+\alpha\closedballset\right)\cap\vec{D}\ne\emptyset\right\} ,\\
\vec{G}_{\alpha}\left(x\right)\triangleq\left\{ v\in\mathbb{R}^{9}:v\in g+\alpha\closedballset,g\in\vec{G}\left((x+\alpha\closedballset)\cap\vec{D}\right)\right\} ,\ \text{for all }x\in\vec{D}_{\alpha},
\end{gather*}
where $\overline{\mathrm{co}}\,X$ denotes the closure of the convex
hull of the set $X$. These perturbations, as illustrated in \cite{GST:09},
include both measurement and modeling error. 

The lack of sensitivity to these perturbations will be expressed in
Theorem~\ref{thm:robustness} by bounding the Lyapunov function by
a class-$\mathcal{KL}$ function. This bound guarantees practical
stability for perturbed solutions starting from arbitrarily large
subsets of the basin of attraction of $\{\pm1\}$ \cite{GST:09}.
\begin{thm}
\label{thm:robustness} Let $V$ be as in (\ref{eq:Lyapunov_Definition}).
Then there exists a class-$\mathcal{KL}$ function $\beta$ such that
for each compact set $K\subset{\unitdualquatset\times\left\{ -1,1\right\} }$
and $\Delta>0$ there exists $\alpha^{*}>0$ such that for each $\alpha\in\left(0,\alpha^{*}\right]$,
the solutions $x_{\alpha}$ from $K$ of the perturbed system $\mathcal{H}_{\alpha}=(\vec{C}_{\alpha},\vec{F}_{\alpha},\vec{D}_{\alpha},\vec{G}_{\alpha})$
satisfy
\[
V\left(x_{\alpha}\left(t,j\right)\right)\le\beta\left(V\left(x_{\alpha}\left(0,0\right)\right),t+j\right)+\Delta,\ \forall\left(t,j\right)\in\mathrm{dom}\,x_{\alpha}
\]
\end{thm}
\begin{proof}
We have that $V$ is a proper indicator function\footnote{Following \cite[p. 145]{GST:12}, a proper indicator function of a
compact set $\mathcal{A}$ in an open set $\mathcal{O}\supseteq\mathcal{A}$
is a continuous function on $\mathcal{O}$ which is positive definite
with respect to $\mathcal{A}$ and such that it tends to infinity
as its argument tends to infinity or to the boundary of $\mathcal{O}$.} of the compact set $\left\{ (1,1),(-1,-1)\right\} $ in ${\unitdualquatset\times\left\{ -1,1\right\} }$.
From \cite[Theorem 14]{GST:09}, there exists a class-$\mathcal{KL}$
function $\beta$ such that for all solutions $x$ of ${\unitdualquatset\times\left\{ -1,1\right\} }$,
\[
V\left(x\left(t,j\right)\right)\le\beta\left(V\left(x\left(0,0\right)\right),t+j\right),\ \forall\left(t,j\right)\in\mathrm{dom}\,x.
\]
From this and from Lemma~\ref{lem:basic-assumptions}, the $\mathcal{KL}$
bound on $V\left(x_{\alpha}\left(t,j\right)\right)$ follows now by
\cite[Theorem 17]{GST:09}. 
\end{proof}
\begin{rem}
Differently from the Lyapunov function proposed in \cite{MST:11}
for its hybrid kinematic controller, the proposed Lyapunov function
(\ref{eq:Lyapunov_Definition}) exploits the non-compactness of $\unitdualquatset$
to be a proper indicator function, enabling the direct proof of Theorem~\ref{thm:robustness}.
\end{rem}

\section{Numerical simulations\label{sec:Numerical-simulations}}

In this section, the effectiveness of the proposed hybrid technique
for robust global stabilization of the rigid body motion is demonstrated
in four different numerical simulations.\footnote{The results of the simulations were computed using MATLAB environment
and the DQ\_robotics toolbox (\url{http://dqrobotics.sourceforge.net/}).} The first simulation considers the robustness of the proposed controller
against chattering. The second simulation shows the influence of the
design parameter $\delta$ in the execution of the controller. The
last two simulations consider a more practical situation using a robotic
manipulator.

We first illustrate the proposed controller global stability and robustness
against measurement noises. To this aim, a simulation is performed
using the hybrid feedback controller (\ref{eq:hybrid control law - twist}),
with hysteresis parameter $\delta=0.3$, and the pure discontinuous
controller (\ref{eq:Han_controller})\textemdash using the same proportional
gain $k=0.08$. For this particular scenario, we assume an initial
condition, $\dq q_{0}=0.001+\imi0.72+\imj0.06+\imk0.69+\dual\left({-}55.15{-}\imi2.52{+}\imj36.71{-}\imk0.59\right)$,
which was chosen arbitrarily, located in the neighborhood of $\eta=0$,
and a measurement noise over $\eta$ set to $\mathcal{N}(0,0.16)$,
that is, a Gaussian random variable with zero mean and $0.16$ variance.
Fig.$\,$\ref{fig:Pure Discontinuous Controller - Chattering - raster version}
illustrates the result from the discontinuous controller (\ref{eq:Han_controller})
whereby one can clearly see the problematic noise influence\textemdash for
instance, the excess of switches causing chattering for up to $20$
seconds and the consequent convergence lag. In contrast, the proposed
hybrid feedback controller ensures a robust performance without chattering
as shown in Fig.$\,$\ref{fig:Convergence }. \newcommand{\optionfig}{1}

To further highlight the absence of chattering and performance improvements
from the hybrid feedback solution (\ref{eq:hybrid control law - twist})\textemdash regardless
the initial and noise conditions and the control parameters\textemdash a
second scenario is devised with initial condition $\dq q_{0}=0.001+\imi0.78+\imj0.57+\imk0.28+\dual\left({-}1.28{+}\imi1.50{-}\imj2.44{+}\imk0.77\right)$
and a zero mean Gaussian measurement noise over $\eta$ with a $0.1$
standard deviation, which was also chosen arbitrarily. The results
illustrating the trajectory of $\eta$ from both the discontinuous
and hybrid controllers\textemdash set with the same control gain,
$k=2$\textemdash are shown in Fig.$\,$\ref{fig:Hybrid vs Discontinuous }.

\newcommand{\figHYBpartea}{%
\scriptsize%
\def\svgwidth{0.445\columnwidth}%
\begingroup%
  \makeatletter%
  \providecommand\color[2][]{%
    \errmessage{(Inkscape) Color is used for the text in Inkscape, but the package 'color.sty' is not loaded}%
    \renewcommand\color[2][]{}%
  }%
  \providecommand\transparent[1]{%
    \errmessage{(Inkscape) Transparency is used (non-zero) for the text in Inkscape, but the package 'transparent.sty' is not loaded}%
    \renewcommand\transparent[1]{}%
  }%
  \providecommand\rotatebox[2]{#2}%
  \ifx\svgwidth\undefined%
    \setlength{\unitlength}{423.798bp}%
    \ifx\svgscale\undefined%
      \relax%
    \else%
      \setlength{\unitlength}{\unitlength * \real{\svgscale}}%
    \fi%
  \else%
    \setlength{\unitlength}{\svgwidth}%
  \fi%
  \global\let\svgwidth\undefined%
  \global\let\svgscale\undefined%
  \makeatother%
  \begin{picture}(1,0.76311311)%
    \put(0,0){\includegraphics[width=\unitlength]{figEx1_DiscVSHybrid_orientation___hybridOnly.pdf}}%
    \put(0.12152485,0.07064925){\color[rgb]{0,0,0}\makebox(0,0)[b]{\smash{$0$}}}%
    \put(0.26782096,0.07064925){\color[rgb]{0,0,0}\makebox(0,0)[b]{\smash{$10$}}}%
    \put(0.41411708,0.07064925){\color[rgb]{0,0,0}\makebox(0,0)[b]{\smash{$20$}}}%
    \put(0.56041319,0.07064925){\color[rgb]{0,0,0}\makebox(0,0)[b]{\smash{$30$}}}%
    \put(0.7067093,0.07064925){\color[rgb]{0,0,0}\makebox(0,0)[b]{\smash{$40$}}}%
    \put(0.85300541,0.07064925){\color[rgb]{0,0,0}\makebox(0,0)[b]{\smash{$50$}}}%
    \put(0.99930153,0.07064925){\color[rgb]{0,0,0}\makebox(0,0)[b]{\smash{$60$}}}%
    \put(0.10396932,0.11734337){\color[rgb]{0,0,0}\makebox(0,0)[rb]{\smash{$0$}}}%
    \put(0.10396932,0.24181459){\color[rgb]{0,0,0}\makebox(0,0)[rb]{\smash{$.2$}}}%
    \put(0.10396932,0.36628319){\color[rgb]{0,0,0}\makebox(0,0)[rb]{\smash{$.4$}}}%
    \put(0.10396932,0.49075441){\color[rgb]{0,0,0}\makebox(0,0)[rb]{\smash{$.6$}}}%
    \put(0.10396932,0.61522301){\color[rgb]{0,0,0}\makebox(0,0)[rb]{\smash{$.8$}}}%
    \put(0.10396932,0.73969424){\color[rgb]{0,0,0}\makebox(0,0)[rb]{\smash{$1$}}}%
    \put(0.02581566,0.35218944){\color[rgb]{0,0,0}\rotatebox{90}{\makebox(0,0)[lb]{\smash{Rotation $\quat r$}}}}%
    \put(0.55841095,0.00512591){\color[rgb]{0,0,0}\makebox(0,0)[b]{\smash{time ($s$)}}}%
    \put(0.17069324,0.17873941){\color[rgb]{0,0,0}\makebox(0,0)[b]{\smash{$\mu_{2}$}}}%
    \put(0.1517078,0.58243042){\color[rgb]{0,0,0}\makebox(0,0)[b]{\smash{$\mu_{1}$}}}%
    \put(0.14842029,0.4973355){\color[rgb]{0,0,0}\makebox(0,0)[b]{\smash{$\mu_{3}$}}}%
    \put(0.3712943,0.71259621){\color[rgb]{0,0,0}\makebox(0,0)[b]{\smash{$\eta$}}}%
  \end{picture}%
\endgroup%
}%
\newcommand{\figHYBparteb}{
\scriptsize\def\svgwidth{0.48\columnwidth}%
\begingroup%
  \makeatletter%
  \providecommand\color[2][]{%
    \errmessage{(Inkscape) Color is used for the text in Inkscape, but the package 'color.sty' is not loaded}%
    \renewcommand\color[2][]{}%
  }%
  \providecommand\transparent[1]{%
    \errmessage{(Inkscape) Transparency is used (non-zero) for the text in Inkscape, but the package 'transparent.sty' is not loaded}%
    \renewcommand\transparent[1]{}%
  }%
  \providecommand\rotatebox[2]{#2}%
  \ifx\svgwidth\undefined%
    \setlength{\unitlength}{420bp}%
    \ifx\svgscale\undefined%
      \relax%
    \else%
      \setlength{\unitlength}{\unitlength * \real{\svgscale}}%
    \fi%
  \else%
    \setlength{\unitlength}{\svgwidth}%
  \fi%
  \global\let\svgwidth\undefined%
  \global\let\svgscale\undefined%
  \makeatother%
  \begin{picture}(1,0.73333333)%
    \put(0,0){\includegraphics[width=\unitlength]{figEx1_DiscVSHybrid_translation_curva3x___hybridonly.pdf}}%
    \put(0.09067393,0.53386138){\color[rgb]{0,0,0}\makebox(0,0)[rb]{\smash{$0$}}}%
    \put(0.09067393,0.59015592){\color[rgb]{0,0,0}\makebox(0,0)[rb]{\smash{$.5$}}}%
    \put(0.09067393,0.64645046){\color[rgb]{0,0,0}\makebox(0,0)[rb]{\smash{$1$}}}%
    \put(0.09067393,0.3157843){\color[rgb]{0,0,0}\makebox(0,0)[rb]{\smash{$0$}}}%
    \put(0.09067393,0.37178378){\color[rgb]{0,0,0}\makebox(0,0)[rb]{\smash{$20$}}}%
    \put(0.09067393,0.42778114){\color[rgb]{0,0,0}\makebox(0,0)[rb]{\smash{$40$}}}%
    \put(0.09067393,0.11352529){\color[rgb]{0,0,0}\makebox(0,0)[rb]{\smash{$0$}}}%
    \put(0.09067393,0.17803112){\color[rgb]{0,0,0}\makebox(0,0)[rb]{\smash{$10$}}}%
    \put(0.09067393,0.24253906){\color[rgb]{0,0,0}\makebox(0,0)[rb]{\smash{$20$}}}%
    \put(0.01688055,0.18752547){\color[rgb]{0,0,0}\rotatebox{90}{\makebox(0,0)[b]{\smash{$p_{3}$}}}}%
    \put(0.01833333,0.40215061){\color[rgb]{0,0,0}\rotatebox{90}{\makebox(0,0)[b]{\smash{$p_{2}$}}}}%
    \put(0.01833333,0.61477527){\color[rgb]{0,0,0}\rotatebox{90}{\makebox(0,0)[b]{\smash{$p_{1}$}}}}%
    \put(0.10174798,0.0678249){\color[rgb]{0,0,0}\makebox(0,0)[b]{\smash{$0$}}}%
    \put(0.24936702,0.0678249){\color[rgb]{0,0,0}\makebox(0,0)[b]{\smash{$10$}}}%
    \put(0.39698606,0.0678249){\color[rgb]{0,0,0}\makebox(0,0)[b]{\smash{$20$}}}%
    \put(0.54460511,0.0678249){\color[rgb]{0,0,0}\makebox(0,0)[b]{\smash{$30$}}}%
    \put(0.69222414,0.0678249){\color[rgb]{0,0,0}\makebox(0,0)[b]{\smash{$40$}}}%
    \put(0.83984319,0.0678249){\color[rgb]{0,0,0}\makebox(0,0)[b]{\smash{$50$}}}%
    \put(0.98746224,0.0678249){\color[rgb]{0,0,0}\makebox(0,0)[b]{\smash{$60$}}}%
    \put(0.54860111,0.00572086){\color[rgb]{0,0,0}\makebox(0,0)[b]{\smash{time ($s$)}}}%
    \put(0.18904892,0.70094132){\color[rgb]{0,0,0}\makebox(0,0)[rb]{\smash{\tiny $\times 10^2$}}}%
  \end{picture}%
\endgroup%
}%
\begin{figure}
\begin{centering}
\begin{tabular}{c||cc}
\multicolumn{2}{c}{%
\scriptsize%
\def\svgwidth{0.445\columnwidth}%
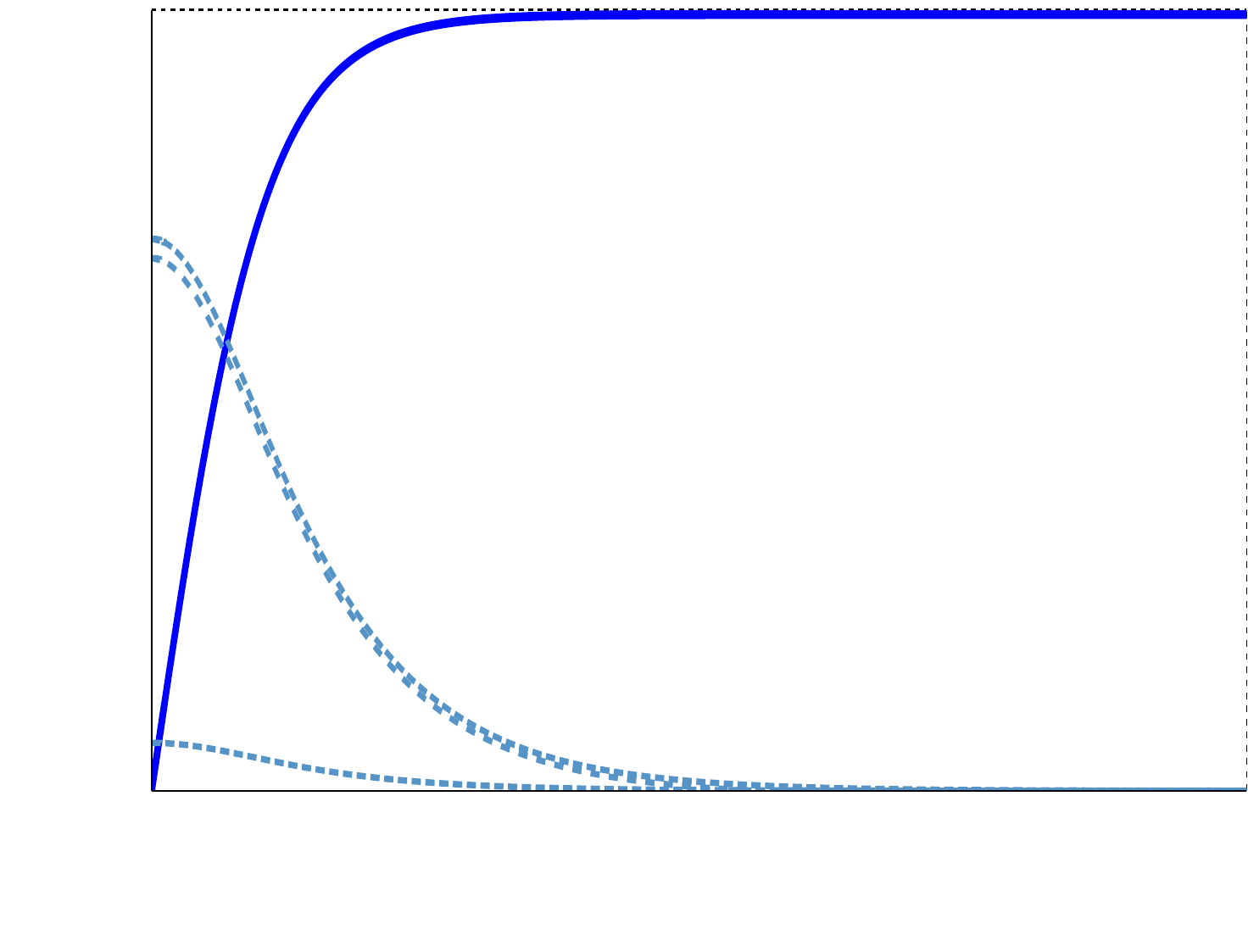%
} & 
\scriptsize\def\svgwidth{0.48\columnwidth}%
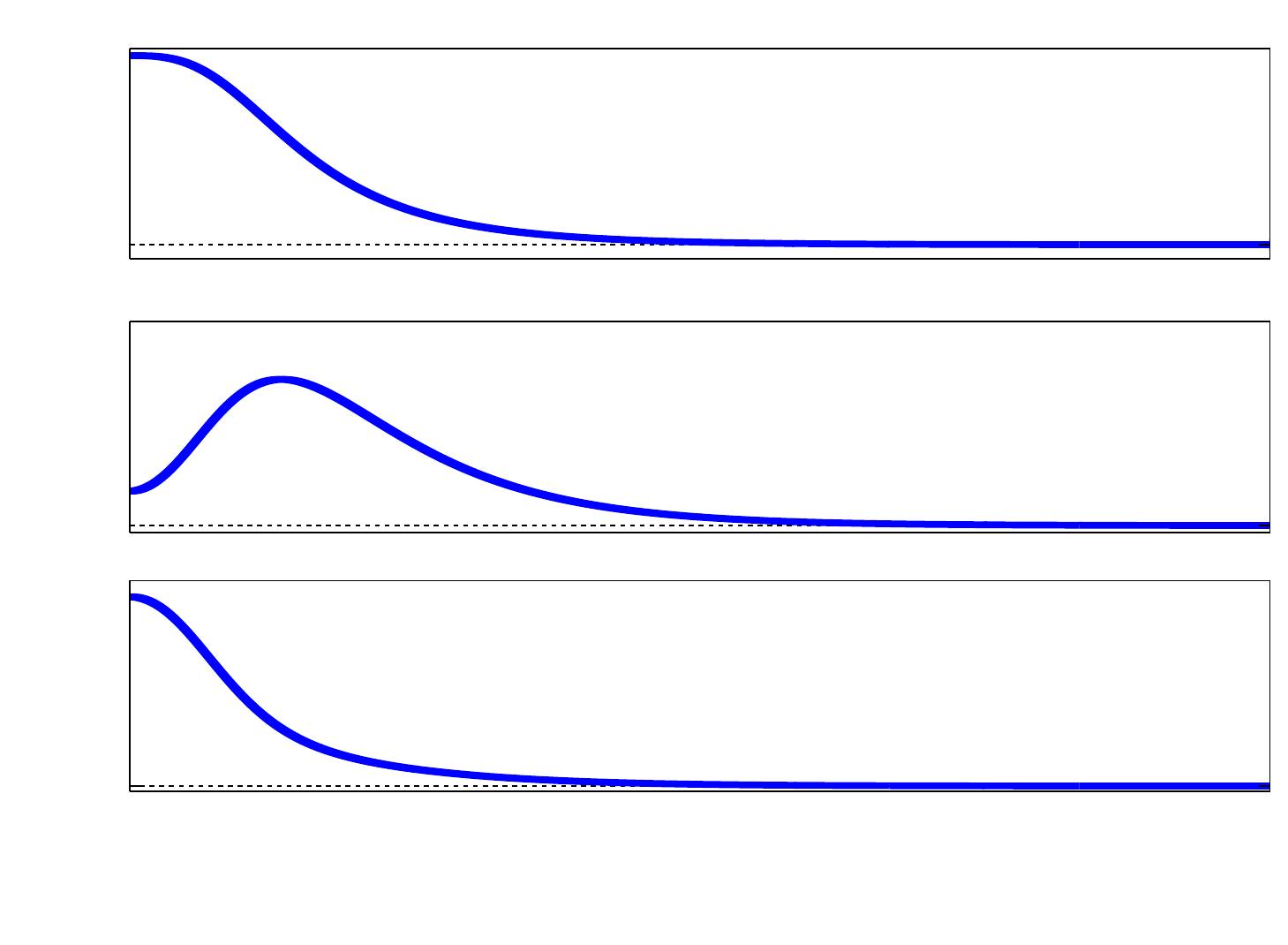%
\tabularnewline
\multicolumn{2}{c}{\vspace{-1pt}$\ $$\ $$\ $$\ $$\ $$\ ${\scriptsize(a)}} & \vspace{-1pt}$\ \ \ ${\scriptsize(b)}\tabularnewline
\end{tabular}
\par\end{centering}
\caption{Numerical example for the hybrid controller: (\emph{a}) Trajectory
of the rotation unit quaternion $\protect\quat r$ in terms of $\eta$
and $\protect\quat{\mu}$ (dashed line). (\emph{b}) Trajectory of
the three-dimensional translation elements $\protect\myvec p=p_{1}\protect\imi+p_{2}\protect\imj+p_{3}\protect\imk$.
\label{fig:Convergence }}
\end{figure}
\begin{figure}
\begin{centering}
\footnotesize%
\def\svgwidth{0.75\columnwidth}%
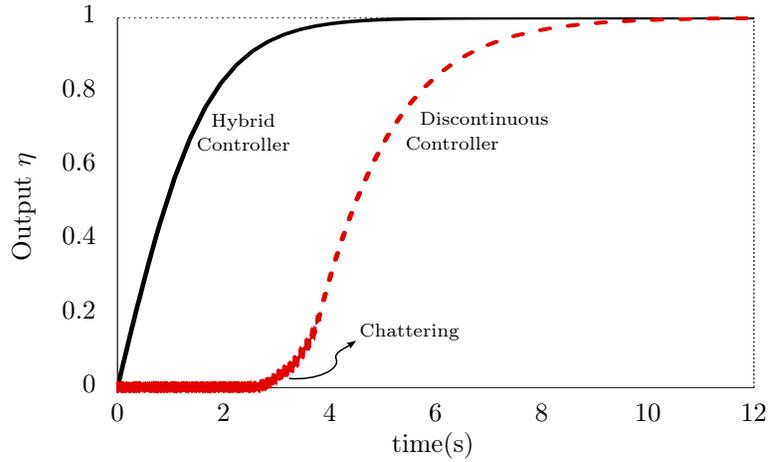%
\par\end{centering}
\caption{Trajectory of $\eta$ with hybrid feedback controller (\ref{eq:hybrid control law - twist})
and discontinuous controller (\ref{eq:Han_controller}) over time.\label{fig:Hybrid vs Discontinuous }}
\end{figure}

To illustrate the influence of the design parameter $\delta$ over
the switches along time of the closed-loop system (\ref{eq:DQ_kinematics}),
a set of simulations is performed using the hybrid controller (\ref{eq:hybrid control law - twist})
with different values for $\delta$. For these simulations, we assume
the same initial condition, control gain, and measurement noise as
defined in the former scenario. As shown in Fig.$\,$\ref{fig:Jumps_by_histeresys}(a),
larger hysteresis parameters yield a smaller number of switches, as
one would expect. As shown in Fig.$\,$\ref{fig:Jumps_by_histeresys}(b),
it is also interesting to highlight that the number of switches tends
to decrease along time as $\eta$ converges to the equilibrium.

\newcommand{\figSwitchespartea}{\scriptsize%
\def\svgwidth{0.46\columnwidth}%
\begingroup%
  \makeatletter%
  \providecommand\color[2][]{%
    \errmessage{(Inkscape) Color is used for the text in Inkscape, but the package 'color.sty' is not loaded}%
    \renewcommand\color[2][]{}%
  }%
  \providecommand\transparent[1]{%
    \errmessage{(Inkscape) Transparency is used (non-zero) for the text in Inkscape, but the package 'transparent.sty' is not loaded}%
    \renewcommand\transparent[1]{}%
  }%
  \providecommand\rotatebox[2]{#2}%
  \ifx\svgwidth\undefined%
    \setlength{\unitlength}{394.803112bp}%
    \ifx\svgscale\undefined%
      \relax%
    \else%
      \setlength{\unitlength}{\unitlength * \real{\svgscale}}%
    \fi%
  \else%
    \setlength{\unitlength}{\svgwidth}%
  \fi%
  \global\let\svgwidth\undefined%
  \global\let\svgscale\undefined%
  \makeatother%
  \begin{picture}(1,1.05470293)%
    \put(0,0){\includegraphics[width=\unitlength]{fig_ex_switches_vs_hysteresis__halfsize.pdf}}%
    \put(0.01081552,0.58389326){\color[rgb]{0,0,0}\rotatebox{90}{\makebox(0,0)[b]{\smash{Number of Switches}}}}%
    \put(0.52854192,0.00384411){\color[rgb]{0,0,0}\makebox(0,0)[b]{\smash{Hysteresis $\delta$}}}%
    \put(0.08984225,0.09225593){\color[rgb]{0,0,0}\makebox(0,0)[b]{\smash{$0$}}}%
    \put(0.38298481,0.09225593){\color[rgb]{0,0,0}\makebox(0,0)[b]{\smash{$0.1$}}}%
    \put(0.67612535,0.09225593){\color[rgb]{0,0,0}\makebox(0,0)[b]{\smash{$0.2$}}}%
    \put(0.96926791,0.09225593){\color[rgb]{0,0,0}\makebox(0,0)[b]{\smash{$0.3$}}}%
    \put(0.07225373,0.15883876){\color[rgb]{0,0,0}\makebox(0,0)[rb]{\smash{$0$}}}%
    \put(0.07225373,0.32593264){\color[rgb]{0,0,0}\makebox(0,0)[rb]{\smash{$1$}}}%
    \put(0.07225373,0.49302651){\color[rgb]{0,0,0}\makebox(0,0)[rb]{\smash{$2$}}}%
    \put(0.07225373,0.66012039){\color[rgb]{0,0,0}\makebox(0,0)[rb]{\smash{$3$}}}%
    \put(0.07225373,0.82721424){\color[rgb]{0,0,0}\makebox(0,0)[rb]{\smash{$4$}}}%
    \put(0.07225373,0.99430813){\color[rgb]{0,0,0}\makebox(0,0)[rb]{\smash{$5$}}}%
    \put(0.23385626,1.01091344){\color[rgb]{0,0,0}\makebox(0,0)[rb]{\smash{$\times 10^3$}}}%
  \end{picture}%
\endgroup%
}%
\newcommand{\figSwitchesparteb}{\scriptsize%
\def\svgwidth{0.46\columnwidth}%
\begingroup%
  \makeatletter%
  \providecommand\color[2][]{%
    \errmessage{(Inkscape) Color is used for the text in Inkscape, but the package 'color.sty' is not loaded}%
    \renewcommand\color[2][]{}%
  }%
  \providecommand\transparent[1]{%
    \errmessage{(Inkscape) Transparency is used (non-zero) for the text in Inkscape, but the package 'transparent.sty' is not loaded}%
    \renewcommand\transparent[1]{}%
  }%
  \providecommand\rotatebox[2]{#2}%
  \ifx\svgwidth\undefined%
    \setlength{\unitlength}{380bp}%
    \ifx\svgscale\undefined%
      \relax%
    \else%
      \setlength{\unitlength}{\unitlength * \real{\svgscale}}%
    \fi%
  \else%
    \setlength{\unitlength}{\svgwidth}%
  \fi%
  \global\let\svgwidth\undefined%
  \global\let\svgscale\undefined%
  \makeatother%
  \begin{picture}(1,1.04210526)%
    \put(0,0){\includegraphics[width=\unitlength]{fig_ex1_switches_vs_time__halfsize.pdf}}%
    \put(0.10026311,0.96906118){\rotatebox{90}{\makebox(0,0)[lb]{\smash{$.02$}}}}%
    \put(0.10026311,0.72564824){\rotatebox{90}{\makebox(0,0)[lb]{\smash{$.10$}}}}%
    \put(0.10026311,0.4684863){\rotatebox{90}{\makebox(0,0)[lb]{\smash{$.18$}}}}%
    \put(0.10026311,0.21829808){\rotatebox{90}{\makebox(0,0)[lb]{\smash{$.26$}}}}%
    \put(0.3925422,0.09268369){\makebox(0,0)[lb]{\smash{0.1}}}%
    \put(0.6662265,0.09268368){\makebox(0,0)[lb]{\smash{0.2}}}%
    \put(0.94013187,0.09268368){\makebox(0,0)[lb]{\smash{0.3}}}%
    \put(0.47609472,0.01138404){\makebox(0,0)[lb]{\smash{time(s)}}}%
    \put(0.0041558,0.45704047){\rotatebox{90}{\makebox(0,0)[lb]{\smash{Hysteresis $\delta$ }}}}%
    \put(0.14666577,0.09268366){\makebox(0,0)[lb]{\smash{0}}}%
  \end{picture}%
\endgroup%
}%
\begin{figure}
\begin{centering}
\begin{tabular}{c||cc}
\multicolumn{2}{c}{\scriptsize%
\def\svgwidth{0.46\columnwidth}%
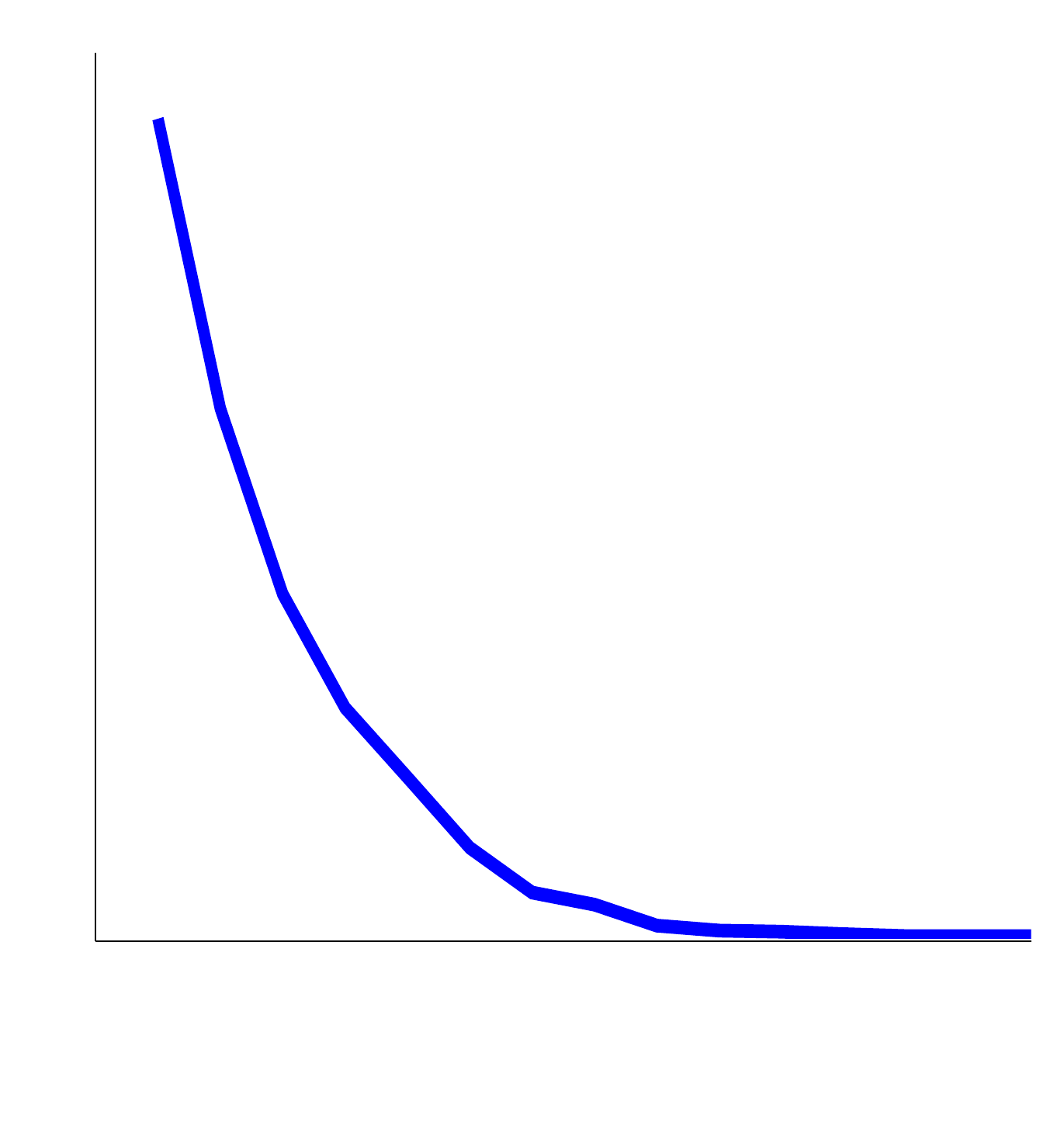%
\vspace{-2pt}} & \scriptsize%
\def\svgwidth{0.46\columnwidth}%
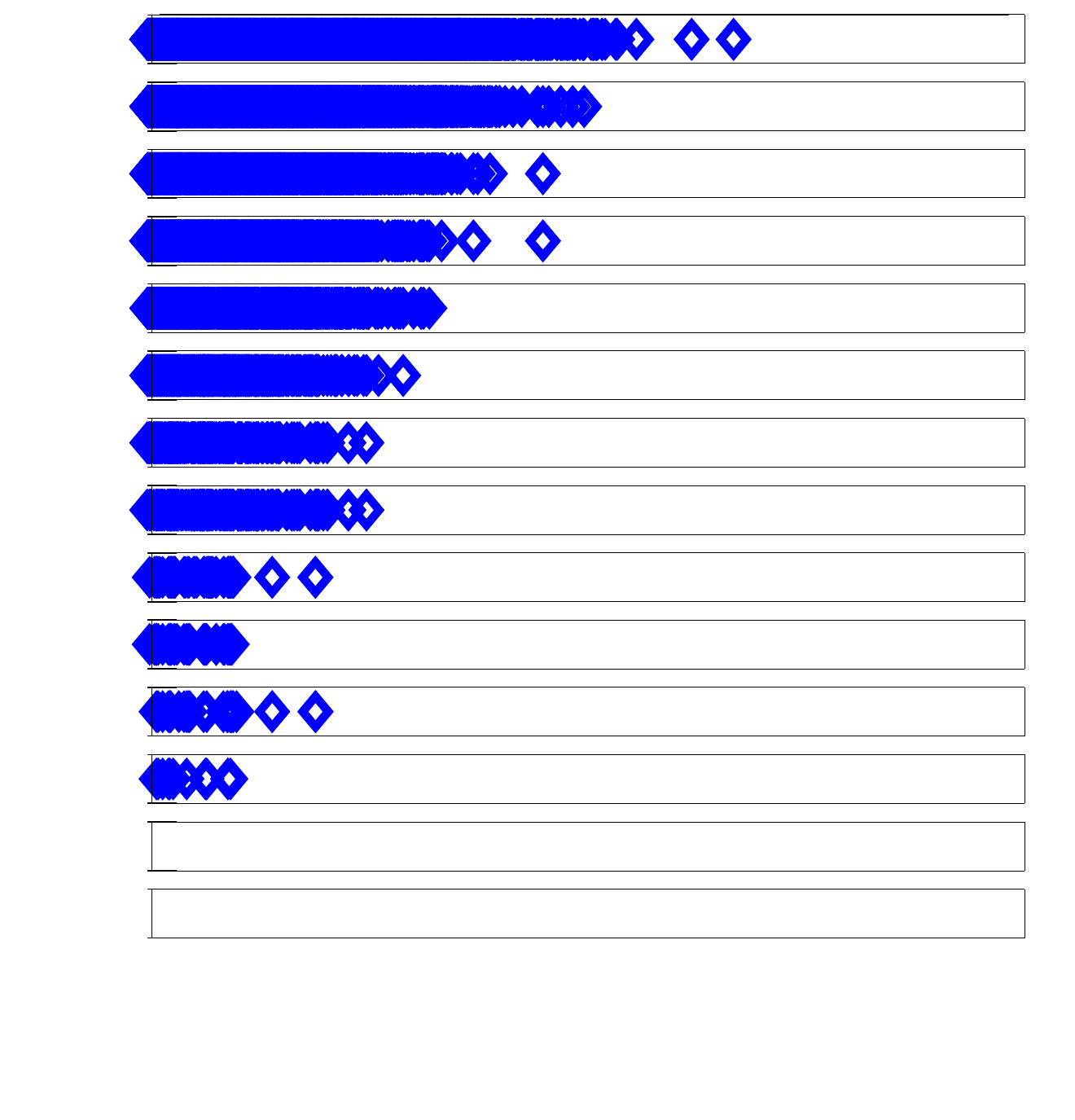%
\vspace{-2pt}\tabularnewline
\multicolumn{2}{c}{\vspace{-1pt}$\ $$\ $$\ $$\ $$\ $$\ ${\scriptsize(a)}} & \vspace{-1pt}$\ \ \ ${\scriptsize(b)}\tabularnewline
\end{tabular}
\par\end{centering}
\caption{The number of switches with regard to the hysteresis parameter $\delta$
is shown in (\emph{a}), while the switches along time $s(t)$ are
illustrated in (\emph{b}) for different values of $\delta$.\label{fig:Jumps_by_histeresys}}
\end{figure}

Moreover, to elucidate the influence of the hysteresis parameter $\delta$
with regard to the unwinding phenomenon, a different scenario is simulated
using (\ref{eq:hybrid control law - twist}) with $\delta=0.15$ and
$\delta=0.95$ and with a proportional gain $k=5$. We assume an initial
state with $\eta$ close to $-1$ and $h=1$. As shown in Fig.$\,$\ref{fig:Unwinding},
very large values of $\delta$ may induce the stabilization to $\eta=1,$
which leads to needless motions and control efforts compared to the
case of $\delta=0.15$. 
\begin{figure}
\begin{centering}
\footnotesize%
\vspace{-20pt}
\def\svgwidth{0.55\columnwidth}%
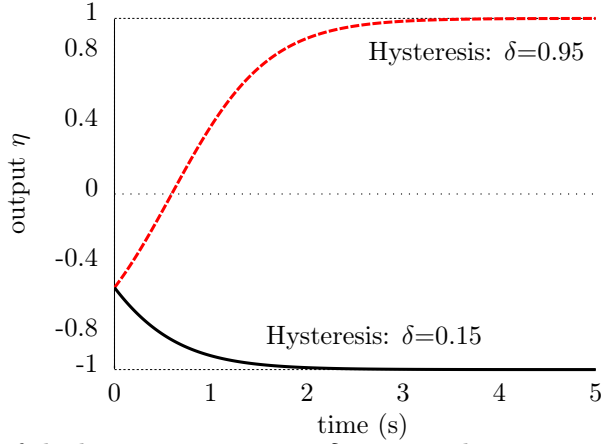%
\par\end{centering}
\caption{Influence of the hysteresis parameter $\delta$ on unwinding\textemdash $\eta$
converges to the farther stable point when the value of $\delta$
increases.\label{fig:Unwinding}}
\end{figure}

Lastly, as a concluding example, and to assess the effectiveness of
the proposed solution in a more practical context, we designed a simple
robot manipulator kinematic control task. To this aim, we considered
a 6-DOF manipulator, the Comau SMART SiX robot, and two simple control
tasks whereby the end-effector of the robot manipulator is regarded
as a rigid body and described within the unit dual quaternion framework.\footnote{Further information on how to describe and map the end-effector's
rigid motion using unit dual quaternions can be found in \cite{Ador:11}.}

\begin{figure}
\begin{centering}
\scriptsize%
\def\svgwidth{0.95\columnwidth}%
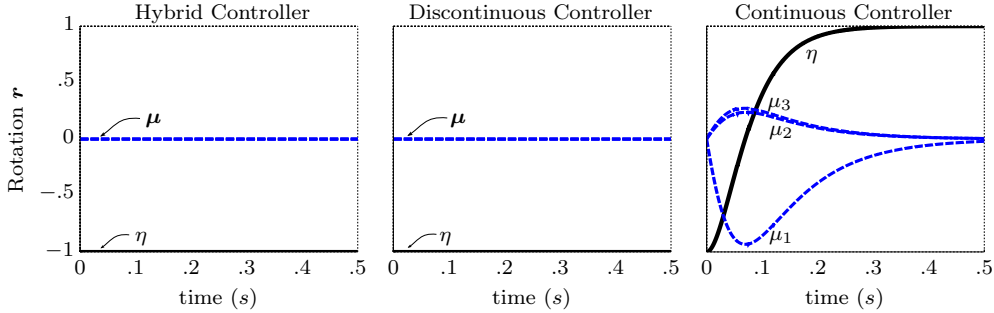%
\par\end{centering}
\caption{Trajectory of the rotation unit quaternion $\protect\quat r$ in terms
of $\eta$ and $\protect\quat{\mu}$ using the proposed hybrid controller
(\emph{left}), the discontinuous controller (\emph{center}), and the
continuous feedback controller (\emph{right}). The unwinding phenomenon
arises only on the continuous feedback controller.\label{fig:manipResult:Unwinding:Rotation}}
\end{figure}

In the first control setting, the end-effector of the manipulator
$\dq q_{m}$, described within unit dual quaternions framework, is
expected to hold the same current configuration\textemdash hence,
the desired pose $\dq q_{d}=\dq q_{m}$\textemdash in the presence
of different sensor readings. In this case, it is rather ordinary
to have readings in the antipodal configuration of the current pose,
that is, $-\dq q_{m}$. To illustrate the behavior of different controllers\textemdash with
gain equally set to $k=5$\textemdash within this particular case,
that is, $\dq q_{d}=-\dq q_{m}$, we set the manipulator to a random
configuration and sought to stabilize the system using a continuous
feedback controller, a discontinuous controller, and the proposed
hybrid controller (with $\delta=0.1$). The simulated result can be
observed in Figs.$\,$\ref{fig:manipResult:Unwinding:Rotation} and
\ref{fig:manipResult:Unwinding:Position}, which illustrate the rigid
motion of the manipulator's end-effector. Clearly, the continuous
feedback controller failed to maintain the same end-effector configuration,
exhibiting the unwinding phenomenon which yields needless motions\textemdash as
observed in Fig.$\,$\ref{fig:manipFigure:Unwinding}.\footnote{Since the discontinuous and hybrid feedback controllers successfully
hold the same end-effector pose, the corresponding trajectories of
the robot were not shown in this figure because they are constant.
A video comparing the trajectories generated by the three different
controllers can be seen in the supplementary material.} Such phenomena could be avoided by simply enforcing a discontinuous
controller or by using the proposed hysteresis-based hybrid control
strategy.

\newcommand{\clab}{orange!90!black!75}
\newcommand{\claf}{blue!50!black!100}

\begin{figure}
\begin{centering}
\scriptsize%
\def\svgwidth{0.50\columnwidth}%
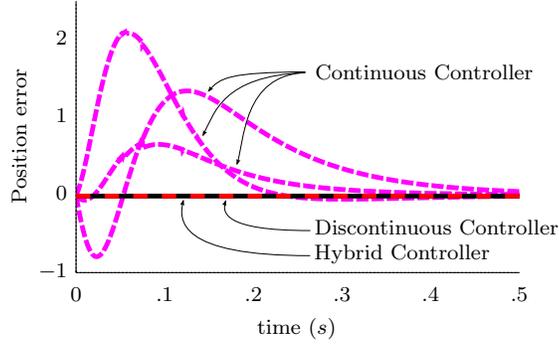%
\par\end{centering}
\caption{Trajectory of the three-dimensional translation error using the hybrid
(\emph{solid line}), the discontinuous (\emph{red dashed line}) and
the continuous controller (\emph{magenta dashed line}).\label{fig:manipResult:Unwinding:Position}}
\end{figure}

\begin{figure}
\centering \scriptsize%
\def\svgwidth{0.93\columnwidth}%
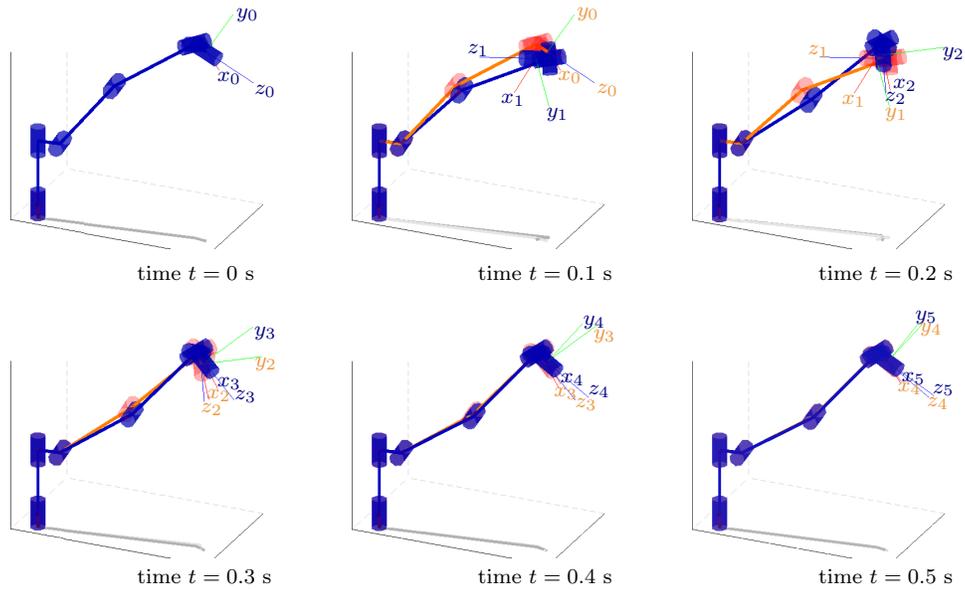 %

\caption{Simulation snapshots of the continuous controller, when the desired
robot pose, represented by $\protect\dq q$, is changed to the same
pose, but now represented by $-\protect\dq q$. The unwinding phenomenon
can be observed in contrast to maintaining the desired pose. In all
snapshots, the light red robot represents the initial robot configuration.\protect\footnotemark\label{fig:manipFigure:Unwinding}}
\end{figure}

Nonetheless, as observed in Fig.$\,$\ref{fig:Hybrid vs Discontinuous },
the discontinuous sign-based approach is particularly sensitive to
measurement noises. Hence, the second control task was devised to
illustrate the behavior of the robot manipulator in the presence of
measurement noises. In this scenario, both controllers were supposed
to take the end-effector pose from an initial pose, represented by
$\dq q_{0}{=}-0.31-\imi0.67+\imj0.67-\imk0.05+\dual\left({-}0.06{-}\imi0.31{-}\imj0.31{+}\imk0.40\right)$
and corresponding to a rotation angle of $(\pi+0.63)$ rad around
the axis $(-\sqrt{2}/2,-\sqrt{2}/2,0)$ followed by a translation
of $(-0.39,-0.29,-1.09)$, to a desired pose, represented by $\dq q_{d}{=}\imi0.707+\imj0.707+\dual\left(0.28{-}\imi0.38{+}\imj0.38{+}\imk0.28\right)$
and corresponding to a rotation angle of $\pi$ rad around the axis
$(\sqrt{2}/2,\sqrt{2}/2,0)$ followed by a translation of $(-0.79,0.00,-1.07)$.
The error between these poses are represented by the dual quaternion
$\dq q_{e}=\dq q_{m}^{*}\dq q_{d}$, where $\dq q_{m}$ is the measured
dual quaternion. In addition, the measurement noise over $\eta$ was
set to $\mathcal{N}(0,0.09)$ and the control gain for both controllers
were set to $k=0.020$\textemdash the hysteresis parameter was set
to $\delta=0.1$. Fig.$\,$\ref{fig:manipResult:compare_Hyb_Discont}
illustrate the rigid motion of the manipulator's end-effector and
the behavior of both controllers. It is easy to see that the problematic
noise influence is restricted to the discontinuous controller\textemdash resulting
in undesired chattering and delaying the closed-loop convergence.
As expected, the proposed hybrid solution ensures robust performance,
that is, a trajectory without chattering. \footnotetext{A video showing the motion of the robot is included in the supplementary material.} 

\begin{figure}
\begin{centering}
\begin{tabular}{cc}
\scriptsize%
\def\svgwidth{0.46\columnwidth}%
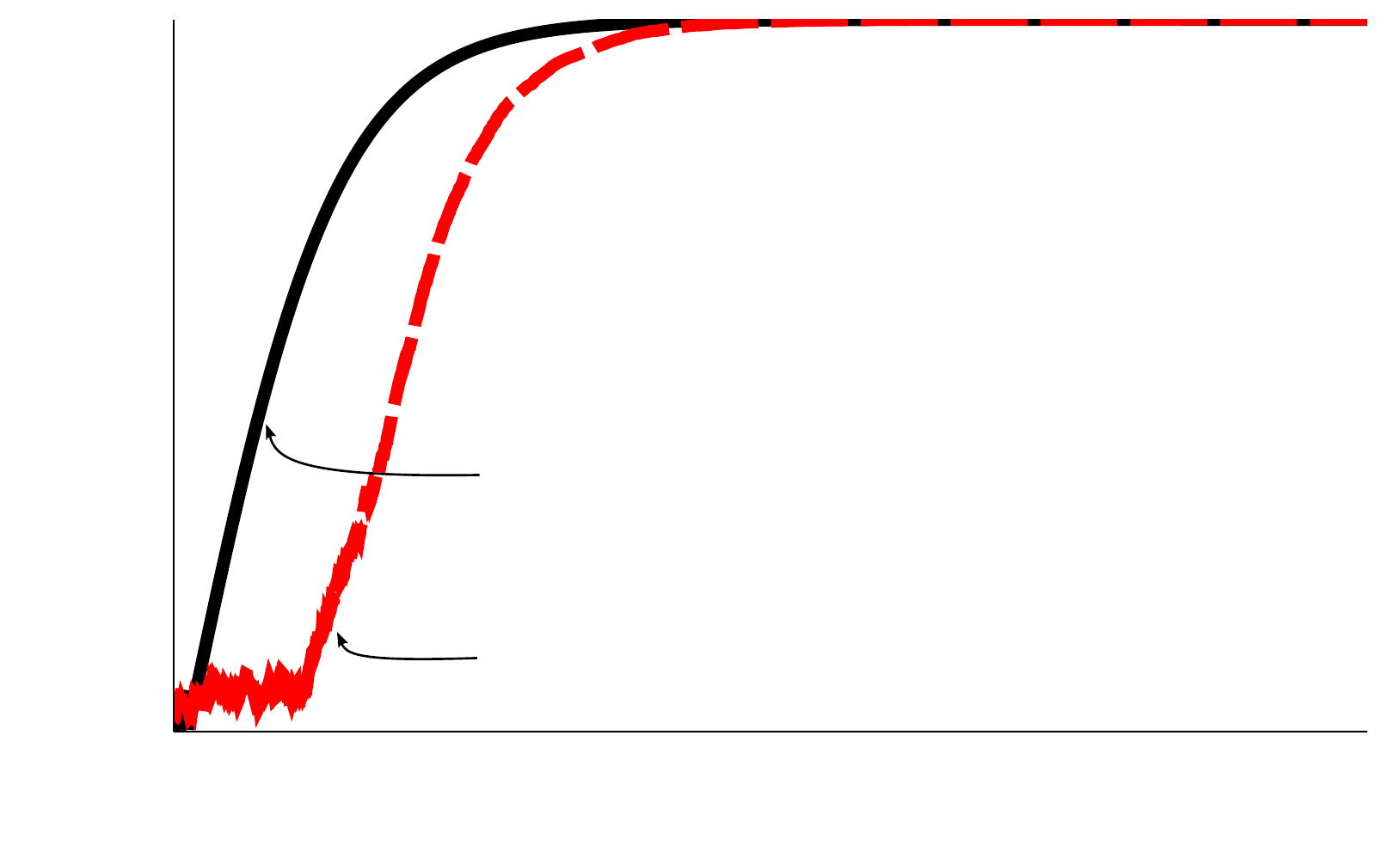 & \scriptsize%
\def\svgwidth{0.46\columnwidth}%
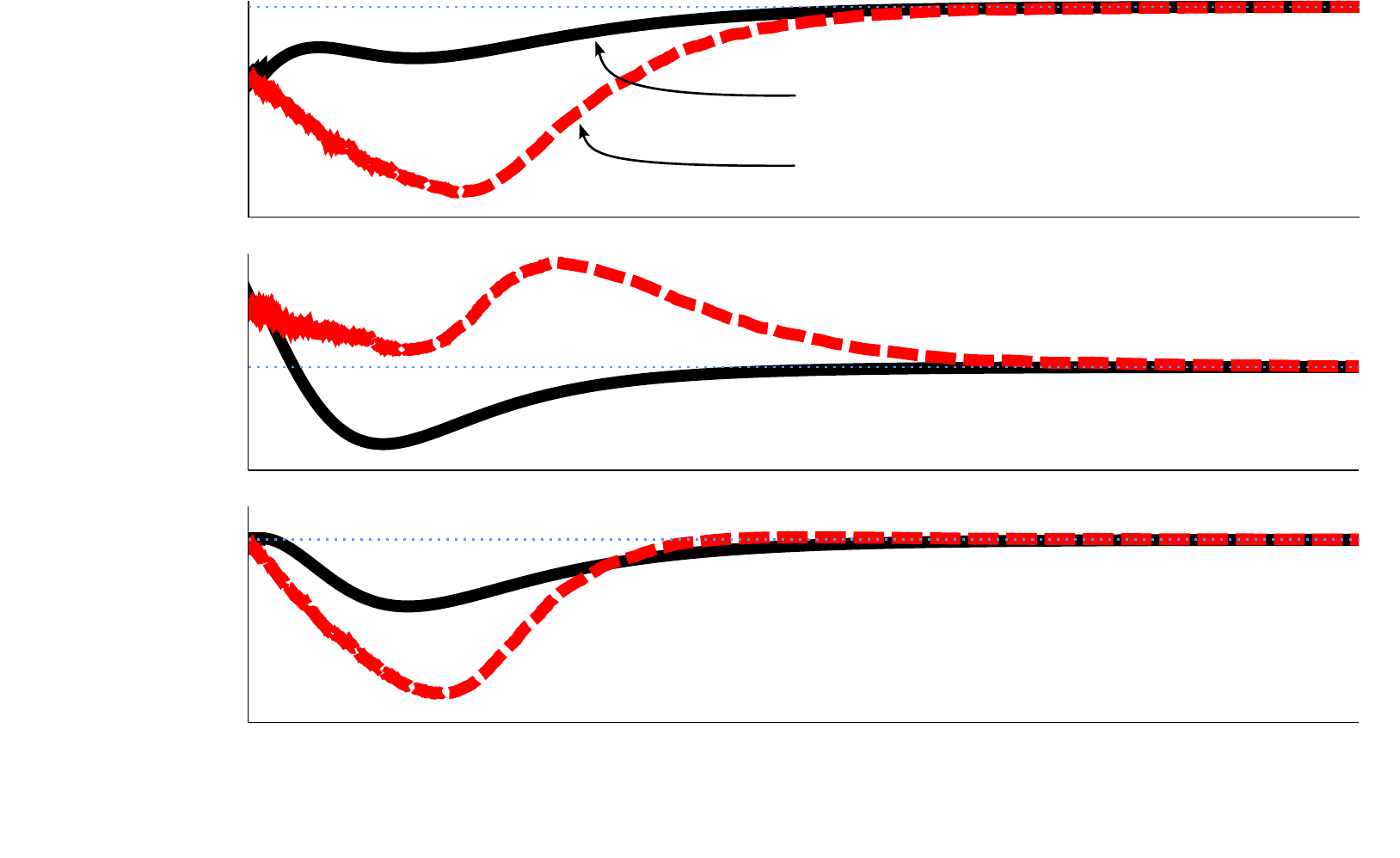\tabularnewline
\end{tabular}
\par\end{centering}
\caption{The left figure shows the trajectory of the rotation error in terms
of the scalar part $\eta$ using the hybrid (\emph{solid line}) and
the discontinuous controller (\emph{red dashed line}). The right figure
shows the trajectory of the three-dimensional translation elements
$\protect\myvec p=p_{1}\protect\imi+p_{2}\protect\imj+p_{3}\protect\imk$
with reference given in dotted blue line.\label{fig:manipResult:compare_Hyb_Discont}}
\end{figure}

\section{Conclusion}

\label{sec:Conclusion}

In this paper, a kinematic controller for the rigid body stabilization
problem was presented. To prove the stability of this controller,
a Lyapunov function that exploits the structure of the group of unit
dual quaternions was proposed. Moreover, this controller was simulated
and compared to another kinematic controller based on dual quaternions
that has recently been presented in the literature. Simulation results
show that the proposed controller is robust against measurement noises
and, different from discontinuous-based feedback controllers, it also
avoids the problem of sensitivity of the global stabilization property
to chattering. The proposed solution was also simulated in a simple
robot manipulator kinematic control task to assess the controller
in a more practical context.

In other scenarios it is possible that the input to the system is
done by torques and forces instead of the generalized velocity. Further
work will aim to incorporate the inertial parameters in the controller
design.

\section*{Acknowledgments}

This work is partially supported by the Coordination for the Improvement
of Higher Education (CAPES), by the Brazilian National Council for
Technological and Scientific Development (CNPq grant numbers are 456826/2013-0
and 312627/2013-0), and by Fundação de Amparo à Pesquisa do Estado
de Minas Gerais (FAPEMIG grant number is APQ-00967-14). We are also
grateful to Paulo Percio Mota Magro for the discussions relevant to
this work.

\section*{Appendix}
\begin{thm}
\label{thm:BB:00} \cite{BB:00} Let $\mathcal{M}$ be a manifold
of dimension $m$ and consider a continuous vector field $f$ on $\mathcal{M}$.
Suppose $\pi:\mathcal{M}\rightarrow\mathcal{L}$ is a vector bundle
on $\mathcal{L}$, where $\mathcal{L}$ is a compact, $r$-dimensional
manifold with $r\le m$. Then there exists no equilibrium of $f$
that is globally asymptotically stable.
\end{thm}

\bibliographystyle{elsarticle-num-names}
\bibliography{abreviada,referencias,extras}

\end{document}